\documentclass[a4paper,11pt]{amsart}
\usepackage{amsfonts}
\usepackage{amssymb}
\usepackage[utf8]{inputenc}
\usepackage{amsmath}
\usepackage{pdflscape}
\usepackage{float}
\usepackage{easyReview}
\usepackage{comment}
\usepackage{graphicx}
\usepackage[colorlinks=true, linkcolor=red, citecolor=blue]{hyperref}
\usepackage[]{epsfig}
\usepackage[]{pstricks}
\usepackage{tikz}
\newcommand{\nwc}{\newcommand}
\nwc\eps{\varepsilon}
\newcommand\nd{\noindent}

\newcommand\I{{\mathcal{I}}}

\newcommand\R{\mathbb{R}}

\nwc\EN{{\mathcal{E}_1}}

\newcommand\ep{\epsilon}
\newcommand\la{\lambda}

\nwc{\dx}{\partial_x}
\nwc{\dy}{\partial_y}

\newcommand\M{{\mathcal{M}}}

\nwc{\hamone}{{\mathcal{H}}}
\nwc{\Imu}{\Sigma}
\nwc{\Iem}{\I}
\nwc{\Gep}{G}
\nwc{\wt}{\widetilde}
\nwc{\IF}{\mathcal F}

 \newtheorem{thm}{Theorem}[section]
 
 \newtheorem{lem}[thm]{Lemma}
 
 \theoremstyle{definition}
 \newtheorem{defn}[thm]{Definition}
 \theoremstyle{remark}
 \newtheorem{rem}[thm]{Remark}
 
 \numberwithin{equation}{section}

\numberwithin{equation}{section}
\makeatletter
\@namedef{subjclassname@2010}{\textup{2020} Mathematics Subject Classification}
\makeatother

\begin{document}

\title[stability of solitary waves]
{Solitary waves for a higher order Boussinesq system: Stability and numerical experiments}

\author[Capistrano-Filho]{Roberto de A. Capistrano--Filho*}

\address{Roberto de A. Capistrano--Filho \\
Department of Mathematics, Universidade Federal de Pernambuco\\
Av. Prof. Moraes Rego, 1235 - Cidade Universitária, Recife - PE, 50670-901, Brasil}

\email{\url{roberto.capistranofilho@ufpe.br}}

\author[Mu\~noz]{Juan Carlos Mu\~noz}


\author[Quintero]{Jos\'e Raul Quintero}

\address{Juan Carlos Mu\~noz and Jos\'e Raul Quintero \\
Department of Mathematics, Universidad del Valle\\
Calle 13, 100-00\\
Cali - Colombia}

\email{\url{juan.munoz@correounivalle.edu.co}}
\email{\url{jose.quintero@correounivalle.edu.co}}




\date{\today}
\thanks{*Corresponding author: \url{roberto.capistranofilho@ufpe.br}}

\begin{abstract}
In this work, we study the nonlinear orbital stability of solitary-wave solutions for a class of higher-order Boussinesq systems with Hamiltonian structure. Using variational methods and the asymptotic connection with generalized fifth-order KdV equations, we establish orbital stability results for a broad family of homogeneous and nonhomogeneous nonlinearities satisfying suitable scaling assumptions. We also perform numerical simulations to investigate the stability criterion associated with the solitary waves. The numerical results suggest that the range of wave velocities leading to orbital stability may be larger than that predicted by the theoretical analysis.
\end{abstract}

\maketitle

\section{Introduction}
In the present work,  we study the nonlinear orbital stability of solitary-wave solutions (travelling waves of finite energy) of the one-dimensional higher-order Boussinesq evolution system
\begin{equation} \left\{
\begin{array}{rl}
\left(I- d\partial^2_x+d_2\partial_x^4\right) u_t  + \eta_x
+c\partial_x^3\eta+c_2\partial_x^5 \eta
&= \partial_x\left(G_1(\eta, \eta_x, \eta_{xx}, u, u_{x}, u_{xx}) \right),
\\ \label{1bbl} \\
\left(I-b\partial_x^2+b_2\partial_x^4\right) \eta_t
+\partial_x u+a\partial_x^3 u +a_2\partial_x^5 u
&=\partial_x\left(G_2(\eta, \eta_x, \eta_{xx}, u, u_{x}, u_{xx}) \right),
\end{array}\right.
\end{equation}
where $\eta=\eta(x,t)$ and $u=u(x,t)$ are real functions,  and the nonlinearity $G=(G_1, G_2)^T$ has the variational structure
\begin{equation*}
\begin{split}
G_1(q, r, z, s, t, w)=& F_q(q, r, s, t)- rF_{qr}(q, r, s, t)- zF_{rr}(q, r, s, t)- tF_{sr}(q, r, s, t)\\&  - wF_{tr}(q, r, s, t),
\end{split}
\end{equation*}
\begin{equation*}
\begin{split}
G_2(q, r, z, s, t, w)=& F_s(q, r, s, t)- rF_{qt}(q, r, s, t)- zF_{rt}(q, r, s, t)- tF_{st}(q, r, s, t)\\&- wF_{tt}(q, r, s, t),
\end{split}
\end{equation*}
where $F$  is a function with some properties (like homogeneity). 

We want to point out that some classical Boussinesq systems were derived by Bona, Chen, and Saut for the first- and second-order approximations to the full two-dimensional Euler equations, to describe the motion of short waves of small amplitude on the surface of an ideal fluid under gravity force (see \cite{BCS02} and \cite{BCS04}).  In the first approximation, the authors obtained a four-parameter family of Boussinesq systems from the two-dimensional Euler equations, known as the $abcd$ Boussinesq system, 
\begin{equation*}
\begin{cases}
\begin{array}{rl}
\left(I- b\partial^2_x\right) \partial_t\eta  + \partial_xw
+a\partial_x^3w
&= -\partial_x\left(\eta w \right),
\\ \label{abcd}
\left(I-d\partial_x^2\right) \partial_tw
+\partial_x \eta+c\partial_x^3 \eta
&=\frac12\partial_x\left(w^2 \right),
\end{array}
\end{cases}
\end{equation*}
with $a+b+c+d=\frac13-\sigma$ ($\sigma\geq 0$ is the surface tension), 
and for the second approximation, an eight-parameter family of Boussinesq systems
$$
\begin{cases}
\begin{array}{rl}
\left(I- d\partial^2_x+d_2\partial_x^4\right) \partial_tw  + \partial_x\eta
+c\partial_x^3\eta+c_2\partial_x^5 \eta
= \partial_x\left(H_1(\eta, \partial_x\eta, \partial^2_x\eta, w, \partial_xw, \partial^2_xw) \right),
\\ \label{8-BS} \\
\left(I-b\partial_x^2+b_2\partial_x^4\right) \partial_t\eta
+\partial_x u+a\partial_x^3 u +a_2\partial_x^5 u
=\partial_x\left(H_2(\eta, \partial_x\eta, \partial^2_x\eta, w, \partial_xw,\partial^2_xw) \right),
\end{array}
\end{cases}
$$
where $H_i$ for $i=1,2$ are given by
\begin{align*}
H_1(\eta,\partial_x \eta,\partial^2_x \eta, w,\partial_x w,\partial^2_xw) &=  \frac12 w^2 -c\partial_x(w\partial_xw)-\eta \partial^2_x\eta-\frac12 \partial_x(w^2) +(c+d)w \partial^2_xw,
\end{align*}
and
\begin{align*}
H_2(\eta, \partial_x\eta,\partial^2_x \eta, w,\partial_x w,\partial^2_xw) &=  -\eta w+b\partial^2_x(\eta w)-\left(a+b-\frac13\right)\eta \partial^2_xw,
\end{align*}
with the constants $a$,  $a_2$, $b$, $b_2$, $c_2$, $d$, and $d_2$ satisfying the following condition
\[
a+b+c+d=\frac13-\sigma,
\]
and also that 
 \[
a_2 -b_2=-\frac12 (\theta^2-\frac13)b+\frac5{24}(\theta^2-\frac15)^2,  \ \
c_2 -d_2=-\frac12 (1-\theta^2)c+\frac5{24}(1-\theta^2)(\theta^2-\frac15),
\]
for $\theta\in [0, 1]$. In the second approximation case for $\sigma=\frac13$, the Boussinesq system still captures the dispersive nature of the water wave problem, which does not hold for the $abcd$-Boussinesq system corresponding to the first approximation. This fact will be crucial to achieve the stability result in our work.

\subsection{Background} It is important to highlight that, as happens for the $abcd$-Boussinesq system and the KdV equation at the traveling wave level (see, for example, 
\cite{CQS-2024}), the higher-order Boussinesq system \eqref{1bbl} at the traveling wave level, depending on the nonlinearity, is related to traveling wave solutions of the fifth-order KdV equation:
\begin{equation}\label{5kw}
\partial_tu + \alpha\partial^3_x u + \beta \partial^5_xu = \partial_x(f(u, \partial_xu, \partial^2_xu)),
\end{equation}
where the nonlinearity takes the variational form
$$f(q, r, s) = F_q(q, r) - rF{qr}(q, r) - sF_{rr}(q, r),$$
for some $C^2$ function $F$, which is not necessarily homogeneous, as discussed by Esfahani and Levandosky in \cite{Esfahani-Levandosky-2021}. 

The fifth-order KdV equation \eqref{5kw} arises in several physical contexts. For the choice $F(u,u_x)=-u^3$, the equation models the evolution of gravity–capillary waves in shallow water, as studied by J. Hunter and J. Scheurle \cite{Hunter-Scheurle-1988} and J. Zufiria \cite{Zufiria-1987}. It also appears in the description of chains of coupled nonlinear oscillators and magneto-acoustic wave propagation in plasmas \cite{Kawahara-1972}.

On the other hand,  P. Olver \cite{Olver-1984} derived equation \eqref{5kw} as a second-order approximation for unidirectional wave propagation in the setting of irrotational motion of an inviscid and incompressible fluid under gravity, describing the interaction of small-amplitude long waves over a shallow horizontal bottom (see also \cite{Craig-Groves-1994}). In this framework, several nonlinear structures arise naturally, including terms of the form $F(u,u_x)=uu_x^2-u^3,$
as well as related nonlinearities such as $G(u,u_x)=-uu_x^2$, $G(u,u_x)=-u^3$, and $G(u,u_x)=u^4+uu_x^2,$ which are associated with perturbations of the Hamiltonian structure of the full water-wave problem. Related models and asymptotic derivations were also investigated by D. Benney \cite{Benney-1977}, W. Craig and M. Groves \cite{Craig-Groves-1994}, and several other authors. For further developments on higher-order water-wave models, we refer the reader to \cite{Benney-1977,Craig-Groves-1994,Hunter-Scheurle-1988,Kawahara-1972,Kichenassamy-Olver-1991,Olver-1984,Ponce-1993,BCS02,BCS04,CNS2011,BCL2005}.

As is well known, the approach developed by Grillakis {\it et al.} in \cite{GSS,GSS1} for the analysis of orbital stability and instability relies on the existence of a Hamiltonian structure, through which travelling waves are characterized as critical points of an associated action functional. For the higher-order Boussinesq system \eqref{1bbl}, the Hamiltonian functional is given formally by
\begin{equation*}\label{ham}
\mathcal{H}\begin{pmatrix}\eta \\ \Phi
\end{pmatrix}
=\frac{1}{2}\int_{\mathbb{R}}
\left(
\eta^2- c (\eta_x)^2+ c_2(\eta_{xx})^2
+u^{2} - a (u_x)^2+a_2(u_{xx})^2
+F(\eta, \eta_x, u,  u_x)
\right)\,dx.
\end{equation*}
Moreover, the system can be written in the Hamiltonian form
\[
\begin{pmatrix}\eta_t  \\
\Phi_t
\end{pmatrix}
=
-\mathcal{J}_{bdb_2d_2}\mathcal{H}'
\begin{pmatrix}\eta \\ \Phi \end{pmatrix},
\]
where
\[
\mathcal{J}_{bdb_2d_2}
=
\partial_x
\begin{pmatrix}
0 &\left(I- b \partial_x^2+b_2\partial_x^4\right)^{-1}\\
\left(I- d\partial_x^2+d_2\partial_x^4\right)^{-1}&0
\end{pmatrix}.
\]
In the particular case \(b=d\) and \(b_2=d_2\), the system possesses a genuine Hamiltonian structure, since the operator \(\partial_x M^{-1}\), with
\[
M^{-1}=\mathcal{J}_{bbb_2b_2},
\]
is skew-symmetric. Furthermore, if the initial data \(\eta(0)\) and \(u(0)\) have zero mean value on \(\mathbb{R}\), then this property is preserved by the flow for as long as the solution exists. In this case, the system also admits the conserved quantity
\(\mathcal Q\), usually referred to as the {\it charge} in the sense of Noether's theorem. Formally, it is defined by
\[
\mathcal{Q}\begin{pmatrix}\eta \\ u \end{pmatrix}
=
-\frac12
\left<
\mathcal{J}_{bb_2}^{-1}
\begin{pmatrix}\eta \\ u\end{pmatrix},
\begin{pmatrix}\eta \\ u \end{pmatrix}
\right>
=
-\int_{\mathbb{R}}
(I-b\partial_x^2+b_2\partial_x^4)\eta u \,dx.
\]
We emphasize that both functionals \(\mathcal H\) and \(\mathcal Q\) are well defined for
\[
(\eta,u)\in H^2(\mathbb R)\times H^2(\mathbb R),
\]
which therefore constitutes the natural energy space for the well-posedness theory and for the construction of travelling-wave solutions.

Under this Hamiltonian framework, travelling waves with speed \(\omega\) correspond to stationary solutions of the modulated system
\[
\begin{pmatrix}\eta_t  \\
\Phi_t
\end{pmatrix}
=
\mathcal{J}_{bb_2}J_\omega'
\begin{pmatrix}\eta \\ \Phi \end{pmatrix},
\]
where the action functional is given by
\[
J_\omega(Y)=\mathcal H(Y)+\omega\mathcal Q(Y).
\]
Equivalently, travelling-wave solutions are characterized as critical points of \(J_\omega\), that is, solutions of
\[
\mathcal H'(Y)+\omega\mathcal Q'(Y)=0.
\]
We recall that the general theory developed in \cite{GSS,GSS1} provides a powerful framework for proving orbital stability of solitary waves in abstract Hamiltonian systems. In this setting, least-energy solitary waves \(Y_\omega\) are characterized as minimizers of the action functional \(J_\omega\). The corresponding stability analysis relies essentially on the positivity properties of the linearized operator
\[
J_\omega''(Y_\omega),
\]
up to at most two distinguished directions, together with the strict convexity of the scalar function,
\[
S_1(\omega)
=
\inf\left\{
J_\omega(Y): Y\in \mathcal M_\omega
\right\},
\]
where \(\mathcal M_\omega\) denotes an appropriate constraint manifold.

For the higher-order Boussinesq system considered here, the spectral characterization of $J_{\omega}''(Y_{\omega})$ within the Grillakis--Shatah--Strauss framework is highly challenging due to the complexity of the associated differential operator and the failure of the range of $\partial_x$ to satisfy the abstract assumptions of the theory (see \cite{BSS,CQS-2024,DBJCb,JQAM2013c,JQ2002a,Q1,JQ2005,JQ2010b,Shatah}). Therefore, we adopt a variational approach based on the minimization properties of $S_1$, following ideas developed in \cite{Shatah,JQ2005,JQ2010b,DBJCb,CQS-2024}. 


From the numerical perspective, the stability and instability of solitary waves in dispersive equations have been extensively investigated in \cite{Bona_numerical,Dougalis,Frutos,BonaDougalis}. These studies combine finite-difference, spectral, and high-order time discretization methods to analyze stability thresholds, nonlinear interactions, and long-time dynamics, providing valuable support to the analytical theory.

\subsection{Main results and heuristic}

In this paper, we study the orbital stability of solitary waves for the higher-order Boussinesq system \eqref{1bbl} in the regime $0<\omega<1$. A key ingredient is the KdV scaling limit as $\omega\to1^{-}$, which links the Boussinesq system to a generalized fifth-order KdV equation and allows us to establish the positivity of $S''(\omega)$ near the limiting velocity. Combined with variational methods, concentration--compactness arguments, and the Grillakis--Shatah--Strauss theory, this yields orbital stability for a broad class of nonlinearities. We also develop a Fourier-based iterative scheme to compute solitary-wave profiles and the stability functional. The numerical results agree with the theory and suggest that the stability region may extend beyond the range covered by the analytical results.

Before we go further,  by a solitary-wave solution, we mean a solution  $(\eta, u)\in H^2(\R)\times H^2(\R)$ for the Boussinesq system (\ref{1bbl}) of the form $\eta(t,x)=  \psi\left({x-\omega
t}\right)$ and $u(x, t)=
v\left({x-\omega t}\right)$. In this case,  the travelling wave profile $(\psi, v)$ must satisfy the system
\begin{equation}
\left\{
\begin{array}{rl}
-\omega\left(v-dv''+d_2 v^{(iv)}\right)+ \psi +c\psi''+c_2 \psi^{(iv)}
-G_1(\psi, \psi_x, \psi_{xx}, v,  v_x, v_{xx})&=0,
\\ \label{trav-eqs}\\
-\omega\left(\psi-b\psi''+b_2 \psi^{(iv}\right) + v+a v'' +a_2 v^{(iv)}
-G_2(\psi, \psi_x, \psi_{xx}, v,  v_x, v_{xx})& =0.
\end{array}\right.
\end{equation}

Owing to the Hamiltonian structure of the system, when
\[
b=d,\qquad b_2=d_2,\qquad a,c<0,\qquad a_2,c_2>0,
\]
travelling-wave solutions correspond to critical points of the action functional $J_\omega$. In particular, ground-state solitary waves are characterized as minimizers of $J_\omega$ under an appropriate constraint. In this case, the existence of solitary-wave solutions for  this higher-order Boussinesq system  (\ref{1bbl}) follows from the variational approach that applies a minimax-type result, since solutions $(\psi,v)$ of the system (\ref{trav-eqs}) are critical points of the functional $J_{\omega}$ given by
$$
J_{\omega}(\psi, v)=\frac12 I_{\omega}(\psi, v)- K(\psi, v),
$$
where the functionals $I_\omega$ and $K$ are defined on the space $
X= H^2(\R)\times H^2(\R)$ by
$$
I_{\omega}(\psi,v)=I_1(\psi,v)+I_{2,{\omega}}(\psi,v),
$$
and 
$$
K(\psi, v)  =  \int_{\mathbb R}F(\psi, \psi_x, v,  v_x)\,dx, 
$$
with
\begin{eqnarray*}
I_1(\psi, v)
& = & \int_{\mathbb R}\left[\psi^2 -c(\psi')^{2} + c_2(\psi'')^{2}+
v^2-a( v')^2+a_2( v'')^2\right]dx,\\
I_{2}(\psi, v)
& = & -2\int_{\mathbb R}\left(\psi-b\psi''+b_2 \psi^{(iv)} \right)v\,dx= -2\int_{\mathbb R}\left(\psi v+b\psi' v'+b_2 v''\psi'' \right)\,dx,\end{eqnarray*}
where $I_{2, \omega}=\omega I_{2}$.   Before we go further,  we observe that
\[
K'(\psi, v)= (G_1(\psi, \psi_x, \psi_{xx}, v,  v_x, v_{xx}),
G_2(\psi, \psi_x, \psi_{xx}, v,  v_x, v_{xx}))^t.
\]
Now,  we set the operators
\begin{align*}
P_{\omega}(\psi, v)& =\left<J'_{\omega}(\psi, v), (\psi, v)\right> = I_\omega(\psi, v)-N(\psi, v), \\
N(\psi, v)& =\left<K'(\psi, v), (\psi, v)\right>,
\end{align*}
and  consider the manifold $\mathcal M_{\omega}\subset H^2(\R)\times H^2(\R)$
\[
\mathcal M_{\omega}=\{(\psi, v)\in H^2(\R)\times H^2(\R): \ P_{\omega}(\psi, v)=0 \},
\]
and define the number
$$
S(\omega)=\inf\{J_\omega(\psi, v): (\psi, v)\in \M_{\omega}\}.
$$
With all this notation and definitions in hand, we state the main result:
\begin{thm}\label{main} Assume that $a+c+2b=0$.
\begin{itemize}
\item[a)] If $0<p<8$ and  $0<p<q$, there exist $\omega$ arbitrarily close to $1^{-}$ such that $\mathcal G_{\omega}$ is stable. 
\item[b)] Suppose that $F(q,r,s, t)$ is a homogeneous function of order $p+2$. If $ 0<p <8$. Then,  there is $\omega*\in (0, 1)$ such that  $\mathcal G_{\omega}$ is stable for $\omega\in (\omega^*, 1)$.
\item[c)]  Suppose that $F(q,r,s, t)$ is a homogeneous function of order $p+2$ and $F(\cdot,r, \cdot, t)$ is a homogeneous function of order $\beta$. If $p+2\beta<8$, then there exists $\omega_0>0$ such that $\mathcal G_{\omega}$ is stable for $\omega\in (\omega_0, 1)$.
\end{itemize}
\nd Here  the set of the ground state solutions $\mathcal G_{\omega}$ is given by
\[
\mathcal G_{\omega}
=
\left\{
(\psi,v)\in X\setminus\{0\}:
J_{\omega}(\psi,v)=S(\omega),
\
P_{\omega}(\psi,v)=0
\right\}.
\]
\end{thm}
The proof of orbital stability is based on a variational characterization of the solitary waves as minimizers of the action functional $J_\omega(\psi,v)$ under the natural constraint $P_\omega(\psi,v)=0.$ The first step consists of showing that the set of ground states is nonempty and that every minimizing sequence is relatively compact in the energy space $X=H^2(\mathbb R)\times H^2(\mathbb R)$, modulo translations. This compactness property is obtained through concentration–compactness arguments, where the KdV scaling regime plays a crucial role. More precisely, when the wave velocity satisfies $|\omega|\to 1^{-}$, the higher-order Boussinesq system asymptotically approaches a generalized fifth-order KdV equation, at the travelling wave level. This asymptotic reduction allows us to identify the limiting variational structure and to exclude vanishing and dichotomy scenarios for minimizing sequences. As a consequence, minimizing sequences converge strongly, up to translations, towards elements of the ground-state manifold $\mathcal G_\omega$. The key fact to prove that the ground-state solution set $\mathcal G_\omega$ is stable under Grillakis' {\sl et al.} variational approach is based on the compactness property of   $\mathcal G_\omega$ and that the limit equation for a special renormalized sequence in  $\mathcal G_\omega$ is precisely a general fifth-order KdV equation, as $\omega \to 1^{-}$. Our stability results are exactly those obtained for S. Levandoski in the case of a fifth-order KdV equation with a $(p+2)$-homogeneous nonlinearity and for Esfahani-Levandosky for a fifth-order KdV equation with a nonlinearity not necessarily homogeneous, but with special scaling properties (see \cite{S-Levand-1999}, \cite{Esfahani-Levandosky-2021}).  

\subsection{Novelties and structure}
The main contribution of this work is the study of orbital stability for a broad class of higher-order Boussinesq systems by combining variational methods, KdV-type asymptotic analysis, and Hamiltonian stability theory. A key observation is that, as $|\omega|\to 1^{-}$, Boussinesq solitary waves converge, after a suitable rescaling, to solitary-wave solutions of a generalized fifth-order KdV equation. This asymptotic connection provides essential information for the stability analysis.
Another novelty is the use of the Grillakis–Shatah–Strauss framework without requiring a detailed spectral analysis of the linearized operator. Instead, the proof relies on variational methods, concentration--compactness arguments, and asymptotic estimates. As a result, we obtain orbital stability for a wider class of nonlinear interactions than those previously considered. Numerical simulations also suggest that the stability region may be larger than that predicted by the analytical theory.

We conclude this introduction by outlining the organization of the paper. In Section \ref{sec2}, we present several preliminary results and review the existence theory for solitary-wave solutions of system \eqref{trav-eqs}, including the existence of ground-state solitary waves obtained through a variational approach. Section \ref{sec3} is devoted to the analysis of the long-wave KdV scaling limit and the convergence of minimizers toward solitary-wave solutions of a limiting generalized fifth-order KdV equation. In Section \ref{sec4}, we establish the positivity of the second derivative of the action functional in the small-amplitude regime and use this property to prove the orbital stability of ground-state solitary waves, thereby establishing the main result of the paper, Theorem \ref{main}. Finally, section \ref{sec5} presents numerical simulations that illustrate and complement the theoretical results. 

\section{Some results}\label{sec2}
As done by Esfahani-Levandosky in \cite{Esfahani-Levandosky-2021}, we state the assumptions on the nonlinear part $F$.

\vspace{0.2cm}

\nd {\bf Assumption of F}.  Define the function $G(w)= w\cdot \nabla F(w)$ for $w\in \R^4$:
\begin{itemize}
\item[(a)] There is $p>0$ such that for $w\cdot \nabla F(w)\geq (p+2) G(w)$ for $w\in \R^4$.
\item[(b)] There exist $0<q_1\leq q_2 <\infty$ and $C>0$ such that for $w\in \R^4$,
\[
|D^2F(w)|\leq C(|w|^{q_1}+ |w|^{q_2}).
\]
\item[(c)] There exist $u, v\in H^2(\R)$ such that
\[
\int_{\R}F(u, u_x, v, v_x)\,dx>0.
\]
\end{itemize}
From these assumptions, R. de A. Capistrano-Filho, J. C. Mu\~noz, and J. R. Quintero in \cite[Lemma 2.1]{RC-JM-JQ-2025} extended Lemmas 2.5 and 2.6 of Esfahani-Levandosky's work in \cite{Esfahani-Levandosky-2021}.
\begin{lem}\label{esf-lev-2021}
Under the assumptions on $F$, for $w, \tilde w\in \R^2$ we have
\begin{itemize}
\item[i)]  $G(w, \tilde w)\geq (p+2)F(w, \tilde w )$.
\item[ii)] $R(\alpha (w, \tilde w))\geq \alpha^{p+2}R(w, \tilde w)$ with $\alpha\geq 1$ and $R(\alpha(w, \tilde w))\leq \alpha^{p+2}R(w, \tilde w), ~~ $with $0<\alpha\leq 1$ for $R=G$ or $R=F$.
\item[iii)] $|F(w, \tilde w)|+|G(w, \tilde w)|+|\left<(w, \tilde w), G(w, \tilde w)\right>|\leq C(|(w,\tilde w)|^{q_1+2}+ |(w, \tilde w)|^{q_2+2})$.
\item[iv)] $N(\alpha w)\geq \alpha^{p+2}N(w)$ with $\alpha\geq 1$ and $N(\alpha w)\leq \alpha^{p+2}N(w)$ with $0<\alpha\leq 1$.
\item[v)] $N(w)\geq (p+2)K(w)$.
\item[vi)] $\left<N'(w), w \right>\geq (p+2)N(w)$.
\item[vii)] $|N(w)|+|K(w)|+|\left<N'(w), w \right>|\leq C\left(||w||_{H^2}^{q_1+2}+ ||w||_{H^2}^{q_2+2}\right)$, ~ for $w\in H^2(\R)\times H^2(\R)$.
\end{itemize}
\end{lem}
\nd We note that, under the assumption that $F$ is $(p+2)$-homogeneous, Euler's identity yields
\[ qF_q(q,r,s,t)+rF_r(q,r,s,t)+ sF_s(q,r,s,t)+tF_t(q,r,s,t)= (p+2)F(q,r,s,t), 
\]
and so that
\begin{align*}
N(u, v)&=\left<K'(\psi, v), (\psi, v)\right>\\
&= \int_{\R}\left(uF_q(u, u_x, v, v_x)+u_xF_r(u, u_x, v, v_x)+vF_s(u, u_x, v, v_x)+v_xF_t(u, u_x, v, v_x)\right)\,dx\\
&=(p+2)K(u, v). 
\end{align*}

Moreover, inspired by the work of Esfahani-Levandosky, we present the following examples of functionals $F:\mathbb{R}^4\to\mathbb{R}$ satisfying the above assumptions:
\begin{itemize}
\item[(a)] 
\[
F(u,v)=F_1(u)+F_1(v),
\]
where
\[
F_1(u)=\frac{|u|^{q+2}}{q+2}+\frac{|u|^{p+2}}{p+2},
\]
with $0<p<q$.

\item[(b)]
\[
F(u,v)=F_1(u)+F_1(v),
\]
where
\[
F_1(u)=\sum_{j=1}^{m} a_j|u|^{q_j+2}
-\sum_{k=1}^{n} b_k|u|^{p_k+2},
\]
with $a_j,b_k>0$ and $0<p_k<q_j$ for all $j,k\in\mathbb{N}$.

\item[(c)]
\[
F(u,\partial_xu,v,\partial_xv)=
F_1(u, v),
\]
where $F_1$ is  $(p+2)$-homogeneous or $F_1(q, \cdot)$ is $\sigma$-homogeneous and $F_1(\cdot, r)$ is $\beta$-homogeneous such that $\sigma+\beta=p+2$.

\item[(d)]
\[
F(u,\partial_xu,v,\partial_xv)
=
F_1(u,\partial_xu)+F_1(v,\partial_xv),
\]
where $F_1$ is $(p+2)$-homogeneous, or $F_1(q, \cdot)$ is $\sigma$-homogeneous and $F_1(\cdot, r)$ is $\beta$-homogeneous such that $\sigma+\beta=p+2$.
\end{itemize}

The following theorem summarizes the results related to the existence of solitary-wave solutions for the general Boussinesq system (see \cite{RC-JM-JQ-2025}):
\begin{thm}
For $0<|w|<\min\left\{1,\frac{-a}{b},\frac{-c}{b},\frac{a_2}{b_2},\frac{c_2}{b_2}\right\}$, we have that:
\begin{itemize}
\item[i.] the functional $I_{\omega}$ is nonnegative and there are positive constants
$M_1(\omega, a, a_2, b, b_2, c, c_2 )$ and $M_2(\omega, a, a_2, b, b_2, c, c_2 )$ such that
\begin{equation}\label{IG1}
M_1 \|(\psi, v)\|_{X}^{2} \leq I_\omega(\psi, v)\leq M_2
 \|(\psi, v)\|_{X}^{2}.
\end{equation}
\item[ii.] $S(\omega)$ exists, is positive and
$$
S(\omega)=\inf\left\{J_{\omega, p}(\psi, v): P_{\omega}(\psi, v)\leq 0 \right\},
$$
where $$J_{\omega, p}=J_\omega-\frac{1}{p+2} P_{\omega}.$$
\item[iii.] If $(\psi_0, v_0)\in X$ is such that $P_{\omega}(\psi_0, v_0)=0$ and $J_{\omega}(\psi_0, v_0)=S(\omega)$, then $(\psi_0, v_0)$ is a weak solution of (\ref{trav-eqs}). 
\item[iv.] Given a  minimizing sequence  $(\psi_{n}, v_n)_n$ of $S(\omega)$,  there exist a subsequence $\left(\psi_{n_{k}}, v_{n_{k}}\right)_k$, a sequence of points $y_{k}\in \mathbb{R}$  and $(\psi_0, v_0)\in X\setminus\{0\}$ such that $(\psi_{n_{k}}(.+y_{k}), v_{n_k}(.+y_{k}))\rightarrow (\psi, v)$ in $X$ and $J_{\omega}(\psi_0, v_0)=S(\omega)$. In other words, $(\psi_0, v_0)$ is a minimizer for $S(\omega)$.
\end{itemize}
\end{thm}

\begin{rem}\label{I+}
For $0<|w|<\min\left\{1,\frac{-a}{b},\frac{-c}{b},\frac{a_2}{b_2},\frac{c_2}{b_2}\right\}$, we note that:
\begin{align*}
I_\omega(\psi,v) =&\int_{\mathbb R}
\left[(\psi-\omega v)^2+\left(1-\omega^2 \right)v^2-c\left(\psi' -\frac{\omega b}{-c} v'\right)^2+\left(-a-\frac{\omega^2 b^2}{-c}
\right)(v')^2\right. \\ & \left. +c_2\left(\psi''- \frac{\omega b_2}{c_2} v'' \right)^2+ \left(a_2-\frac{\omega^2 b_2^2}{c_2}\right)(v'')^2\right]dx \geq 0,
\end{align*}
and
\begin{align*}
I_\omega(\psi,v)
=&\int_{\mathbb R}
\left[(v-\omega \psi)^2+\left(1-\omega^2 \right)\psi^2-a\left(v' -\frac{\omega b}{-a} \psi'\right)^2+\left(-c-\frac{\omega^2 b^2}{-a}
\right)(\psi')^2\right. \\ & \left. +a_2\left(v''- \frac{\omega b_2}{a_2} v'' \right)^2+ \left(c_2-\frac{\omega^2 b_2^2}{a_2}\right)(\psi'')^2\right]dx \geq 0.
\end{align*}
\end{rem}

We point out that minimizing sequences for $S(\omega)$ are bounded in $X$. In fact, let $(\psi_n, v_n)\in X$ be such that $I_{\omega}(\psi_n, v_n)=N(\psi_n,v_n)$ and that $J_{\omega}(\psi_n, v_n)=S(\omega) +o(1)$, as $n\to \infty$. Now, we note that
\begin{align*}
\frac12 I_{\omega}(\psi_n, v_n)
&\leq J_{\omega}(\psi_n, v_n)+\frac1{p+2}N(\psi_n, v_n) \\
&\leq S(\omega)+\frac1{p+2}I_{\omega}(\psi_n, v_n)+o(1),
\end{align*}
which implies that $||(\psi_n, v_n)||_X$ is bounded, since
\[
\left( \frac{p}{2(p+2)}\right)I_{\omega}(\psi_n, v_n)\leq S(\omega)+o(1).
\]
Therefore,  the existence of solitary-wave solutions follows by using the Lions' Concentration-Compactness Principle (\cite{Lions-1984a,Lions-1984b}) applied to the measure defined by the density $\rho$
$$
\rho(\psi,v) =\psi^2 - c(\psi')^{2} + c_2(\psi'')^{2}+
v^2-a( v')^2+a_2( v'')^2 -2\omega\left(\psi v+b\psi' v'+b_2 v''\psi'' \right).
$$
We note from the proof of inequality (\ref{IG1}) that
$$
M_1 \sigma(\psi, v) \leq \rho(\psi, v)\leq M_2 \sigma(\psi, v),
$$
where $\sigma$ is given by
\[
\sigma(\psi, v)=\psi^2+v^2+(\psi')^2+(v')^2 +(\psi'')^2+(v'')^2.
\]
If we had a minimizing sequence $(\psi_n, v_n)_n\subset X$ for $S(\omega)$, meaning that $P(\psi_n, v_n)=0$ and $J_{\omega}(\psi_n, v_n) \to S(\omega)$, we consider  the sequence of measures $(\nu_n)_n$ given by
\[
\nu_n(A)= \int_A \rho(\psi_n, v_n) \,dx,
\]
which is such that $\nu_n(\R) \to L$, according to the previous discussion. 

Now, remember that the set of ground states is given by
$$
\mathcal G_{\omega}=\left\{(\psi, v)\in X\setminus\{0\}: J_{\omega}(\psi, v)=S(\omega),  \  \mbox{and} \ P_{\omega}(\psi, v)=0\right\}.
$$
We note also that $\mathcal G_{\omega}$ can be characterized as
$$
\mathcal G_{\omega}=\left\{(\psi, v)\in X\setminus\{0\}: J_{\omega, p}(\psi, v)=S(\omega),  \  \mbox{and} \ P_{\omega}(\psi, v)\leq0\right\}.
$$
R. de A. Capistrano-Filho {\sl et al.} established in \cite[Lemma 2.10]{RC-JM-JQ-2025} the following result, which follows from the variational characterization of $S(\omega)$. 
\begin{lem}\label{estd}
Let $0<|w|<\min\left\{1,-\frac{a}{b},-\frac{c}{b},\frac{a_2}{b_2},\frac{c_2}{b_2}\right\}$. Then:
\begin{itemize}
\item[i.] If $0<\omega<1$, then $S({\omega})$ is uniformly bounded.
\item[ii.] If $0<{\omega}_1 <\omega_2<1$ and $(\psi,v)\in\mathcal{G}_{\omega}$, then we have $I_{2,{\omega}}(\psi,v)$ is uniformly bounded on $[{\omega}_1,{\omega}_2]$.
\item[iii.] For $0< \omega <1$ and $(u, v)\in \mathcal G_{\omega}$, we have $S'(\omega)=I_{2}(u, v)$.
\end{itemize}
\end{lem}


\section{The fifth-order KdV scaling for the higher-order Boussinesq  system}\label{sec3}
In this section, we present some auxiliary lemmas that will be used to prove the main result of this article. We show that a suitably renormalized family of solitary-wave solutions of the generalized Boussinesq system converges to a nontrivial solitary-wave solution of a generalized fifth-order KdV equation as the wave speed $|\omega|$ approaches $1^-$. 

To this end, for $\epsilon>0$, we choose the wave speed according to
\[
\omega^{2}=1-\epsilon^{\frac{4}{p+1}}.
\]

Now, given a pair $(\psi,v)\in X$, we introduce the KdV scaling
\begin{equation}\label{kdv-s}
\psi(x)=\alpha z(y), \qquad v(x)=\alpha w(y),
\qquad
y=\epsilon^{\frac{1}{p+1}}x,
\end{equation}
where $\alpha>0$ will be chosen appropriately. We see directly that 
\begin{equation*}
J_{\omega}(\psi, v)= \epsilon^{\frac{3}{p+1}} \alpha^2 J^\epsilon(z,w), \ \ \ J^\epsilon(z,w)=  \frac12I^\epsilon(z,w)- K^\epsilon(z,w), 
\end{equation*}
where $I^\epsilon(z,w)$  and $K^\epsilon(z,w)$ are given by
\begin{align*}
I^\epsilon(z,w)&=I^{1,\epsilon}(z,w)+I^{2,\epsilon}(z,w),\\
I^{1,\epsilon}(z,w)&=\int_{\mathbb{R}}\left(\epsilon^{-\frac{4}{p+1}}z^2- c \epsilon^{-\frac{2}{p+1}}(z')^2+ c_2(z'')^{2}+ \epsilon^{-\frac{4}{p+1}}w^2- a \epsilon^{-\frac{2}{p+1}}(w')^2+a_2( w'')^2\right)\,dy, \\
I^{2,\epsilon}(z,w)&=-2 \omega \int_{\mathbb{R}}(\epsilon^{-\frac{4}{p+1}}zw+b\epsilon^{-\frac{2}{p+1}}z' w'+b_2 z'' w'')\,dy,\\
K^{\ep}(z,w)&=\frac{1}{\epsilon^{\frac{4}{p+1}}\alpha^2} \int_{\mathbb R}F\left(\alpha(z, \epsilon^{\frac{1}{p+1}} z', w, \epsilon^{\frac{1}{p+1}} w')\right)\,dy.
\end{align*}
We also define the corresponding  functional $N^{\ep}$ as
\begin{equation*}
N^{\ep}(z,w)=\frac{1}{\epsilon^{\frac{4}{p+1}} \alpha}\int_{\mathbb R}(z, \epsilon^{\frac{1}{p+1}} z', w, \epsilon^{\frac{1}{p+1}} w')\nabla F\left(\alpha(z, \epsilon^{\frac{1}{p+1}} z', w, \epsilon^{\frac{1}{p+1}}w')\right)\,dy.
\end{equation*}
\begin{rem}
We compute  the functionals in some specific cases that correspond to appropriate generalizations of the nonlinearities in Esfahani-Lewandosky's paper \cite{Esfahani-Levandosky-2021}, mainly,
\[
F_1(v)=\frac{|v|^{q+2}}{q+2}+\frac{|v|^{p+2}}{p+2}, \ \ F_2(u, v)=u^{p+1} v, \ \ F_3(v)=|v|^{p+2}+v_x^{2r}v^s,
\]
under the assumption for the third case
\[
p(2-r)=2(2r+s-2).
\]
With $F_4(v, v_x)$, we denote a $(p+2)$-homogeneous function $F_4(v, v_x)$ and with $F_5(v, v_x)$, we denote  a $(p+2)$-homogeneous function such that $F_{{5}}(q, \cdot)$ is $\sigma$-homogeneous  and $F_{{5}}(\cdot, r)$ is $\beta$-homogeneous with $\sigma +\beta=p+2$. 

For cases $F_1$, $F_3$, $F_4$ and $F_5$, we consider $\wt F_j(q, r. s, t)=F_j(q, r)+F_j(s, t)$. For $\alpha=\ep^{\frac{4}{p(p+1)}}$, we have
\begin{align*}
K_1(\psi, v)&=\epsilon^{\frac{3}{p+1}} \alpha^2 \int_{\R} \left(\frac{\ep^{\frac{4(q-p)}{p(p+1)}}(|z|^{q+2}+|w|^{q+2})}{q+2}+\frac{(|z|^{p+2}+|w|^{p+2})}{p+2}\right)dy:=\epsilon^{\frac{3}{p+1}} \alpha^2 K_1^{\ep}(z, w), \\
K_2(\psi, v)&=\epsilon^{\frac{3}{p+1}} \alpha^2 \int_{\R} z^{p+1}wdy:=\epsilon^{\frac{3}{p+1}} \alpha^2 K_2(z, w), \\
K_3(\psi, v)&=\epsilon^{\frac{3}{p+1}} \alpha^2 \int_{\R} \left(|z|^{p+2}+|w|^{p+2}+z_y^{2r}z^s+w_y^{2r}w^s\right)dy:=\epsilon^{\frac{3}{p+1}} \alpha^2 K_3(z, w),\\
K_4(\psi, v)&=\epsilon^{\frac{3}{p+1}} \alpha^2 \int_{\R} \widetilde{ F}_4(z,\epsilon^{\frac{1}{p+1}} z', w, \epsilon^{\frac{1}{p+1}} w')\,dy:=\epsilon^{\frac{3}{p+1}}\alpha^2 K^{\ep}_4(z, w), \\
K_5(\psi, v)&=\epsilon^{\frac{3}{p+1}}  \epsilon^{\frac{\beta}{p+1}} \alpha^2 \int_{\R} \wt F_5(z, z', w, w')\,dy:=\epsilon^{\frac{3}{p+1}}  \epsilon^{\frac{\beta}{p+1}} \alpha^2 K_5(z, w).
\end{align*}
For these $K_j$, we have that
\begin{align*}
N_1(\psi, v)&=\epsilon^{\frac{3}{p+1}} \alpha^2 \int_{\R} \left(\ep^{\frac{4(q-p)}{p(p+1)}}(|z|^{q+2}+|w|^{q+2})+|z|^{p+2}+|w|^{p+2}\right)dy
\\ &:=\epsilon^{\frac{3}{p+2}} \alpha^2 N_1^{\ep}(u, z), \\
N_2(\psi, v)&=(p+2)\epsilon^{\frac{3}{p+1}} \alpha^2 K_2(z, w).\\
N_3(\psi, v)&=\epsilon^{\frac{3}{p+1}} \alpha^2 \int_{\R} \left((p+2)(|z|^{p+2}+|w|^{p+2})+(2r+s)((z')^{2r}z^s+(w')^{2r}w^s)\right)dy\\
&:=\epsilon^{\frac{3}{p+1}} \alpha^2 N_3^{\ep}(z, w),\\
N_4(\psi, v)&=(p+2)\epsilon^{\frac{3}{p+1}} \alpha^2 K_4^{\ep}(z, w)
\\
N_5(\psi, v)&=(p+2)\epsilon^{\frac{3}{p+1}}\epsilon^{\frac{\beta}{p+1}} \alpha^2 K_5(z, w).
\end{align*}  
\end{rem}
From the properties in  Lemma \ref{esf-lev-2021} (generalization of  Esfahani-Lewandosky's paper \cite{Esfahani-Levandosky-2021}), we see that the following limits exist:
\begin{equation}\label{lim-K-N}
\lim_{\ep\to 0^+}K^{\ep}(z^{\ep},w^{\ep}), \ \ \lim_{\ep\to 0^+}N^{\ep}(z^{\ep},w^{\ep}).
\end{equation}

\medskip

Now, for $0<|\omega|<\min\left\{1, \frac{-a}{b}, \frac{-c}{b},\frac{a_2}{b_2}, \frac{c_2}{b_2}\right\}$,  we know that there is a family $\left((\psi_{\omega},v_{\omega})\right)_{\omega}\,$ such that
$$
J_{\omega}(\psi_{\omega},v_{\omega})=S(\omega),  \  \  P_{\omega}(\psi_{\omega},v_{\omega})=0.  \ \
$$
Then, for $\alpha=  \epsilon^{\frac{4}{p(p+1)}}$, if we denote
$$
S^\epsilon:=\inf\left\{J^\epsilon(z,w) \   :    \  (z,w)\in X\,\,\,\,
\text{with} \ \  \  P^{\ep}(z,w)=0 \right\}, \ \ P^{\ep}=  I^{\ep}-N^{\ep},
$$
there is a correspondent family $((z^\epsilon,w^\epsilon))_\epsilon$ such that
\begin{equation}\label{Ic=eIe}
S^\epsilon=J^\epsilon(z^\epsilon,w^\epsilon), \ \  \
P^{\ep}(z^\epsilon,w^\epsilon)=0,
 \ \  \  S(\omega)=\epsilon^{\frac{3p+8}{p(p+1)}}S^\epsilon\,.
\end{equation}
We also have that in the distributional sense, the family  $((z^\epsilon,w^\epsilon))_\epsilon$ is a solution of the system
\begin{equation}\label{trav-ep}
\begin{cases}
-\omega(\ep)\left(
w
-b\epsilon^{\frac{2}{p+1}}w''
+b_2\epsilon^{\frac{4}{p+1}} w^{(iv)}
\right)
+z
+c\epsilon^{\frac{2}{p+1}}z''
+c_2\epsilon^{\frac{4}{p+1}} z^{(iv)}
-\alpha^{-1}G_{j,\ep,\alpha}(z,w)
&=0,
\\
-\omega(\ep)\left(
z
-b\epsilon^{\frac{2}{p+1}}z''
+b_2\epsilon^{\frac{4}{p+1}} z^{(iv)}
\right)
+w
+a\epsilon^{\frac{2}{p+1}}w''
+a_2\epsilon^{\frac{4}{p+1}} w^{(iv)}
-\alpha^{-1}G_{j,\ep,\alpha}(z,w)
&=0,
\end{cases}
\end{equation}
where $G_{j, \ep, \alpha}$ for $j=1, 2$ is given by
\[
G_{j, \ep, \alpha}(z, w)= G_j\left(\alpha\left(z, \epsilon^{\frac{1}{p+1}}z', \epsilon^{\frac{2}{p+1}}z'', w,  \epsilon^{\frac{1}{p+1}} w', \epsilon^{\frac{2}{p+1}}w''\right)\right).
\]
In the case where $F(q, r, s,t)$ is a $(p+2)$-homogeneous function, we see that 
\begin{equation}\label{Gj-H1}
G_{j, \ep, \alpha}(z, w)= \alpha^{p+1}G_j\left(z, \epsilon^{\frac{1}{p+1}}z', \epsilon^{\frac{2}{p+1}}z'', w,  \epsilon^{\frac{1}{p+1}} w', \epsilon^{\frac{2}{p+1}}w''\right).
\end{equation}
Moreover, if $F(q, r, s,t)$ is a $(p+2)$-homogeneous function and homogeneous of degree $\beta$ with respect to the variable $(r, t)$, then we see that
\begin{equation}\label{Gj-H2}
G_{j, \ep, \alpha}(z, w)= \alpha^{p+1} \ep^{\frac{\beta}{p+1}}G_j\left(z, z', z'', w,  w', w''\right).
\end{equation}

We observe that in the case $a+c+2b=0$, or equivalently $\sigma=\frac13$, the family $((z^\epsilon,w^\epsilon))_\epsilon$ is related to solitary-wave solutions for a generalized fifth-order KdV equation as
\[
\ep\to0
\qquad
(|\omega|\to1^{-}).
\]
To see this, let us define on $H^2(\R)$ the functionals
\begin{equation}\label{Je}
\Gamma_{\ep}(w)=J^{\ep}(\omega(\ep)w,w),
\end{equation}
and
\begin{equation}\label{We}
\wp_{\ep}(w)=P^{\ep}(\omega(\ep)w,w).
\end{equation}
We also define
\begin{equation}\label{JJe}
\mathcal J_{\ep}
=
\inf_{w\in H^2(\R)}
\left\{
\Gamma_{\ep}(w):
\wp_{\ep}(w)=0
\right\}.
\end{equation}
Now, if $v\in H^2(\R)$ satisfies $\wp_{\ep}(v)=0,$ then $P^{\ep}(\omega(\ep)v,v)=0,$ and therefore
\[
\Gamma_{\ep}(v)
=
J^{\ep}(\omega(\ep)v,v)
\geq
S^{\ep}
=
\ep^{-\frac{3p+8}{p(p+1)}}S(\omega(\ep)).
\]
Consequently,
\[
\mathcal J_{\ep}
\geq
S^{\ep}
=
\ep^{-\frac{3p+8}{p(p+1)}}S(\omega(\ep))
=
\frac{p}{2(p+2)}
\ep^{-\frac{3p+8}{p(p+1)}}
I_{\omega(\ep)}
(\psi_{\omega(\ep)},v_{\omega(\ep)}).
\]

On the other hand, under the condition $a+c+2b=0$, we define the limiting functionals
\[
\begin{split}
I_0(w)
&=
\int_{\R}
\left(
w^2
+
(c_2-2b_2+a_2)(w'')^2
\right)dx,
\\
K_0(w)
&=
\lim_{\ep\to0^{+}}
K^{\ep}(\omega(\ep)w,w),
\\
N_0(w)
&=
\lim_{\ep\to0^{+}}
N^{\ep}(\omega(\ep)w,w),
\\
\wp_0(w)
&=
I_0(w)-N_0(w),
\\
J_0(w)
&=
\frac12 I_0(w)-K_0(w),
\end{split}
\]
whose existence follows from \eqref{lim-K-N}. We then consider the limiting minimization problem
\begin{equation*}
\mathcal J_0
=
\inf_{w\in H^2(\R)}
\left\{
J_0(w):
\wp_0(w)=0
\right\}.
\end{equation*}
Using the condition $a+c+2b=0$, we conclude that for every $v\in H^2(\R)$,
\begin{align*}
\lim_{\ep\to0^{+}}
\Gamma_{\ep}(v)
&=
\lim_{\ep\to0^{+}}
\Bigg[
\frac12
\int_{\R}
\Big(
v^2
+
((-c-2b)\omega(\ep)^2-a)(v')^2
\\
&\qquad\qquad
+
((c_2-2b_2)\omega(\ep)^2+a_2)(v'')^2
\Big)\,dx
-
K^\ep(\omega(\ep)v,v)
\Bigg]
\\
&=
\frac12
\int_{\R}
\left(
v^2
+
(c_2-2b_2+a_2)(v'')^2
\right)\,dx
-
K_0(v)
\\
&=
J_0(v).
\end{align*}

Now, let $v\in H^2(\R)\setminus\{0\}$ be such that $\wp_0(v)=0$. Then necessarily $N_0(v)>0$. Hence, for $\ep>0$ sufficiently small,
$N^\ep(\omega(\ep)v,v)>0.$ Moreover, for $\beta>0$,
\[
\wp_\ep(\beta v)=P^\ep(\omega(\ep)\beta v,\beta v)=\beta^2
\left(I^\ep(\omega(\ep)v,v)-
N^\ep(\beta\omega(\ep)v,\beta v)\right).
\]
Using the properties of $N$ in Lemma \ref{esf-lev-2021}, the function $\la (\beta)=\wp_\ep(\beta v)$  is positive for $\beta<1$, but small and negative for $\beta>1$, but large. Then, there is $\beta(\ep)>0$ such that 
\[
\wp_\ep(\beta(\ep)v)=0.
\]
Furthermore,
\[
\lim_{\ep\to0^{+}}\frac{\beta^2
I^\ep(\omega(\ep)v,v)}{N^\ep(\beta\omega(\ep)v,\beta v)}=\frac{\beta_0^2
I_0(v)}{N_0(\beta_0 v)}=1, 
\]
where $\beta_0= \lim_{\ep\to0^{+}}\beta(\ep)$. Using the properties of $N_0$ inherited from $N$ in Lemma \ref{esf-lev-2021}, we see that $\beta_0=1$. For $\beta_0>1$, we get
\[
1<\frac{I_0(v)}{\beta^p_0 N_0( v)}=\frac1{\beta_o^p}
\]
since $\wp_0(v)=0$, which is a contradiction. The same argument shows that $\beta_0<1$ is not possible. So $\beta_0=1$. Therefore,
\[
\Gamma_\ep(\beta(\ep)v)
=
J^\ep(\beta(\ep)\omega(\ep)v,\beta(\ep)v)
\geq
S^\ep.
\]
Passing to the limit,
\[
\limsup_{\ep\to0}S^\ep
\leq
J_0(v).
\]
Taking the infimum over all admissible $v$, we conclude that
\begin{equation}\label{est2}
\limsup_{\ep\to0}S^\ep
\leq
\mathcal J_0.
\end{equation}

Finally, for $\ep>0$ sufficiently small,
\begin{equation}\label{est1}
\frac{p\ep^{-\frac{3p+8}{p(p+1)}}}{2(p+2)}
I_{\omega(\ep)}
(\psi_{\omega(\ep)},v_{\omega(\ep)})
\leq
S^\ep
\leq
2\mathcal J_0.
\end{equation}
Now, we are ready to establish the behavior as $\ep \to 0$ of the functional \eqref{We} and the number \eqref{JJe}.
\begin{lem}\label{lim-1}
Assume that $a+c+2b=0.$ Let $(\ep_j)_j$ be a sequence such that $\ep_j\to0^{+}$, and let $\left((z^{\ep_j},w^{\ep_j})\right)_{j\in\mathbb N}$ be a family of minimizers of $J^{\ep_j}$ under the constraint
\[
P^{\ep_j}(z,w)=0.
\]
Then there exists a subsequence (still denoted the same), a sequence $(x_j)_j\subset\R$, and $(z_0,w_0)\in H^2(\R)\times H^2(\R),$
such that the translated sequence
\[
(\tilde z^{\ep_j},\tilde w^{\ep_j})
:=
(z^{\ep_j}(\cdot+x_j),w^{\ep_j}(\cdot+x_j)),
\]
converges strongly in $H^2(\R)\times H^2(\R)$ to $(z_0,w_0)$. Moreover, $z_0=w_0$.
\end{lem}

\begin{proof}
From the previous discussion, we know that
\[
S^{\ep_j}
=
J^{\ep_j}(z^{\ep_j},w^{\ep_j}),
\qquad
P^{\ep_j}(z^{\ep_j},w^{\ep_j})=0,
\]
and
\[
\limsup_{j\to\infty}S^{\ep_j}\leq 2\mathcal J_0.
\]
Moreover, from \eqref{est1},
\[
\frac{p}{2(p+2)}
I^{\ep_j}(z^{\ep_j},w^{\ep_j})
\leq
S^{\ep_j}
\leq 2
\mathcal J_0.
\]
Hence, $\left(
I^{\ep_j}(z^{\ep_j},w^{\ep_j})
\right)_j
$
is uniformly bounded.

We now apply the Lions' concentration--compactness principle to the sequence $\left(
(z^{\ep_j},w^{\ep_j})
\right)_j.$
Using the positivity of $I^{\ep_j}$ (see Remark \ref{I+}), we define the positive measures
\[
d\nu_j=\rho_j(x)\,dx,
\]
where
\begin{equation*}
\begin{split}
\rho_j (z, w)=&  \epsilon_j^{-\frac{4}{p+1}}\left (z-\omega({\ep_j}) w\right )^2+ w^2 + \epsilon_j^{-\frac{2}{p+1}}|c|\left(z' -  \frac{b\omega({\ep_j})}{|c|} w'\right)^2
\\&+ c_2\left(z''  -  \frac{b\omega({\ep_j})}{c_2}w''\right)^2  +\epsilon_j^{-\frac{2}{p+1}}\left(\frac{ac-b^2 \omega^2({\ep_j})}{|c|}\right) (w')^2+\left(\frac{a_2c_2-b_2^2 \omega^2({\ep_j})}{c_2}\right) \left (w''\right )^2\, ,
\end{split}
\end{equation*}
which is the integrand of $I^{\ep_j} \left (z^{\ep_j} , w^{\ep_j} \right )$.

We note that, arguing exactly as in \cite[Lemmas 2.5 and 2.6]{RC-JM-JQ-2025}, the Vanishing and Dichotomy alternatives can be excluded. Therefore, the Compactness alternative holds. Consequently, there exists a subsequence (still denoted the same) and a sequence $(x_j)_j\subset\R$ such that for every $\gamma>0$, there exists $R>0$ satisfying
\[
\int_{B_R(x_j)}d\nu_j
\geq
L_0-\gamma,
\qquad
j=1,2,\dots,
\]
where
\[
L_0
:=
\limsup_{j\to\infty}
I^{\ep_j}(z^{\ep_j},w^{\ep_j}).
\]

Define
\[
\tilde z^{\ep_j}(x)
=
z^{\ep_j}(x+x_j), \qquad \text{and}
\qquad
\tilde w^{\ep_j}(x)
=
w^{\ep_j}(x+x_j).
\]
Then
\begin{equation}\label{B1}
\int_{\R\setminus B_R(0)}
\tilde\rho_j(x)\,dx
\leq
2\gamma,
\end{equation}
where
\[
\tilde\rho_j(x)
=
\rho_j(\tilde z^{\ep_j}(x),\tilde w^{\ep_j}(x)).
\]
Since $\left(
(\tilde z^{\ep_j},\tilde w^{\ep_j})
\right)_j$
is bounded in $H^2(\R)\times H^2(\R),$ the Sobolev embedding theorem ensures that there exist $(z_0,w_0)\in H^2(\R)\times H^2(\R)$ and a subsequence, still denoted by the same index, such that, as $j\to\infty$,
\begin{align*}
(\tilde z^{\ep_j},\tilde w^{\ep_j})
&\rightharpoonup
(z_0,w_0)
\quad
\text{weakly in }
H^2(\R)\times H^2(\R),
\\
(\tilde z^{\ep_j},\tilde w^{\ep_j})
&\to
(z_0,w_0)
\quad
\text{strongly in }
H^1_{\mathrm{loc}}(\R)\times H^1_{\mathrm{loc}}(\R),
\\
(\tilde z^{\ep_j},\tilde w^{\ep_j})
&\to
(z_0,w_0)
\quad
\text{a.e. in }\R^2.
\end{align*}
Using \eqref{B1} and arguing as usual in concentration--compactness theory, we conclude that
\[
(\tilde z^{\ep_j},\tilde w^{\ep_j})
\to
(z_0,w_0)
\quad
\text{strongly in }
L^2(\R)\times L^2(\R).
\]

Moreover, since
\[
1-\omega^2(\ep_j)
= \ep_j^{\frac4{p+1}} , 
\]
and
\[
I^{\ep_j}(z^{\ep_j},w^{\ep_j})
\]
remains uniformly bounded, we deduce that
\[
\|\tilde z^{\ep_j}-\omega(\ep_j)\tilde w^{\ep_j}\|_{L^2}^2
\to0.
\]
Since
\[
\omega(\ep_j)\to1,
\]
it follows that
\[
\tilde z^{\ep_j}-\tilde w^{\ep_j}\to0
\quad
\text{in }L^2(\R),
\]
and therefore
\[
z_0=w_0.
\]

Now, by weak lower semicontinuity,
\[
I_0(z_0)
\leq
\liminf_{j\to\infty}
I^{\ep_j}(\tilde z^{\ep_j},\tilde w^{\ep_j})
=
L_0.
\]
On the other hand, by weak convergence in $H^2(\R)$,
\begin{align*}
&I^{\ep_j}
(\tilde z^{\ep_j}-z_0,\tilde w^{\ep_j}-w_0)
\\
&=
I^{\ep_j}
(\tilde z^{\ep_j},\tilde w^{\ep_j})
-
I^{\ep_j}(z_0,w_0)
+
o(1).
\end{align*}
Passing to the limit and using the previous inequality, we obtain
\[
\limsup_{j\to\infty}
I^{\ep_j}
(\tilde z^{\ep_j}-z_0,\tilde w^{\ep_j}-w_0)
\leq0.
\]
Since the quadratic form is nonnegative,
\[
\lim_{j\to\infty}
I^{\ep_j}
(\tilde z^{\ep_j}-z_0,\tilde w^{\ep_j}-w_0)
=
0.
\]
Finally, using the coercivity of $I^{\ep_j}$ together with the already established $L^2$ convergence, we conclude that
\[
(\tilde z^{\ep_j},\tilde w^{\ep_j})
\to
(z_0,w_0)
\quad
\text{strongly in }
H^2(\R)\times H^2(\R).
\]
This completes the proof.
\end{proof}

The next result establishes the convergence properties of the functionals
defined in \eqref{Je}--\eqref{JJe}.

\begin{lem}\label{lim-2}
Assume that $c+a+2b=0$. Let $\{(z^\ep,w^\ep)\}_\ep$ be a family of minimizers associated with
$S^\ep$, namely,
\[
J^\ep(z^\ep,w^\ep)=S^\ep,
\qquad
P^\ep(z^\ep,w^\ep)=0.
\]
Then, up to translations and subsequences, there exists
$w_0\in H^2(\R)\setminus\{0\}$ such that
\[
z^\ep-w^\ep\to 0
\quad \text{in }L^2(\R),
\]
and
\[
w^\ep\to w_0
\quad \text{strongly in }H^2(\R).
\]
Moreover,
\begin{align*}
\mathcal J_0
=
\lim_{\ep\to0^+}S^\ep
=
\lim_{\ep\to0^+}\Gamma_\ep(w^\ep),
\end{align*}
and
\[
\wp_0(w_0)=0.
\]
\end{lem}
\begin{proof}
From \eqref{est2}, we already know that
\[
\mathcal J_0
\geq
\limsup_{\ep\to0^+}\mathcal J_\ep
\geq
\limsup_{\ep\to0^+}S^\ep .
\]
On the other hand, by Lemma \ref{lim-1}, after extraction of a subsequence
and suitable translations, there exists
$w_0\in H^2(\R)\setminus\{0\}$ such that
\[
(z^{\ep_j},w^{\ep_j})\to (w_0,w_0)
\quad\text{strongly in }H^2(\R)\times H^2(\R).
\]
Since
\[
P^{\ep_j}(z^{\ep_j},w^{\ep_j})=0,
\]
passing to the limit and using the convergence properties of
$I^{\ep_j}$ and $N^{\ep_j}$, we obtain
\[
\wp_0(w_0)=0.
\]
Therefore, by the definition of $\mathcal J_0$,
\[
\mathcal J_0\leq J_0(w_0).
\]
Now, using the strong convergence together with the definitions of
$\Gamma_\ep$, $K_0$, and $J_0$, we infer that
\begin{align*}
J_0(w_0)
&=
\lim_{j\to\infty}
\Gamma_{\ep_j}(w^{\ep_j}) \\
&=
\lim_{j\to\infty}
J^{\ep_j}(\omega(\ep_j)w^{\ep_j},w^{\ep_j}).
\end{align*}
Since
\[
z^{\ep_j}-w^{\ep_j}\to0
\quad\text{in }L^2(\R),
\]
and $\omega(\ep_j)\to1$, we also obtain
\[
J^{\ep_j}(\omega(\ep_j)w^{\ep_j},w^{\ep_j})
-
J^{\ep_j}(z^{\ep_j},w^{\ep_j})
\to0.
\]
Hence,
\[
J_0(w_0)
=
\lim_{j\to\infty}
J^{\ep_j}(z^{\ep_j},w^{\ep_j})
=
\lim_{j\to\infty}S^{\ep_j}.
\]
Combining the previous inequalities yields
\[
\mathcal J_0
\leq
\liminf_{\ep\to0^+}S^\ep
\leq
\limsup_{\ep\to0^+}S^\ep
\leq
\mathcal J_0,
\]
which proves that
\[
\lim_{\ep\to0^+}S^\ep=\mathcal J_0.
\]

Finally, since
\[
\Gamma_\ep(w^\ep)
=
J^\ep(\omega(\ep)w^\ep,w^\ep),
\]
and
\[
J^\ep(\omega(\ep)w^\ep,w^\ep)
-
J^\ep(z^\ep,w^\ep)\to 0,
\]
we conclude thus that
\[
\lim_{\ep\to0^+}\Gamma_\ep(w^\ep)=\mathcal J_0,
\]
and the proof is completed.
\end{proof}

Now, we are in a position to see that a translated subsequence of the renormalized sequence $(w^\ep)$ converges weakly to a $w_0$ that satisfies the system \eqref{trav-eqs}. i.e., $w_0$ is a weak solution of a general fifth KdV-type equation (\ref{5kw}) with $\nu=0$ and $\beta=a_2+c_2-2b_2$.
\begin{lem}\label{conver}
Assume that $c+a+2b=0.$ Let $\{\ep_j\}_j$ be any sequence such that $\ep_j\to0^+$. Then, up to translations and extraction of a subsequence, there exists a nontrivial function $w_0\in H^2(\R)$ such that
\begin{equation*}
(z^{\ep_j},w^{\ep_j})
\to
(w_0,w_0)
\qquad
\text{strongly in }
H^2(\R)\times H^2(\R).
\end{equation*}
Moreover, $w_0$ satisfies, in the distributional sense,
\begin{equation*}
w_0+\left(c_2+a_2-2b_2\right)\partial_x^4 w_0
=
f(w_0,\partial_x w_0,\partial_x^2 w_0),
\end{equation*}
where
\[
f(q,r,z) = H_q(q,r) - rH_{rq}(q,) -zH_{rr}(q,r),
\]
and
\[
K_0(u)=\int_{\R}H(u,u_x)\,dx,
\]
for some functional $H$ satisfying Esfahani-Lewandosky conditions in \cite{Esfahani-Levandosky-2021}. Consequently, $w_0$ is a solitary-wave solution of the limiting generalized fifth-order KdV equation.
\end{lem}

\begin{proof}
The convergence statement follows directly from Lemma \ref{lim-1}.
In particular,
\[
(z^{\ep_j},w^{\ep_j})
\to
(w_0,w_0)
\quad
\text{strongly in }
H^2(\R)\times H^2(\R),
\]
with $w_0\neq0$. Moreover, by Lemma \ref{lim-2},
\[
\wp_0(w_0)=0,
\qquad
J_0(w_0)=\mathcal J_0.
\]
Therefore, $w_0$ is a minimizer of $J_0$ constrained to the manifold
\[
\mathcal N_0
=
\{w\in H^2(\R)\setminus\{0\}:\wp_0(w)=0\}.
\]
By the Lagrange multiplier principle, there exists $\lambda\in\R$
such that
\[
J_0'(w_0)=\lambda\,\wp_0'(w_0).
\]
Now, using the homogeneity assumptions on the nonlinearity, we obtain
\[
\langle N_0'(w_0),w_0\rangle
\geq
(p+2)N_0(w_0).
\]
Since
\[
\wp_0(w_0)=I_0(w_0)-N_0(w_0)=0,
\]
it follows that
\begin{align*}
\langle \wp_0'(w_0),w_0\rangle
&=
2I_0(w_0)-\langle N_0'(w_0),w_0\rangle \\
&\leq
2I_0(w_0)-(p+2)N_0(w_0) \\
&=
-p\,I_0(w_0)<0.
\end{align*}
Hence, $\lambda=0$, and therefore $J_0'(w_0)=0$. 

Finally, computing the Euler--Lagrange equation associated with $J_0$, we conclude that $w_0$ satisfies
\[
w_0+\left(c_2+a_2-2b_2\right)\partial_x^4 w_0
=
f(w_0,\partial_x w_0,\partial_x^2 w_0),
\]
in the distributional sense, giving the proof.
\end{proof}

We end this section with a technical result that will be clever in the stability analysis in the homogeneous case. 
\begin{lem}\label{lem:H2-bound}
Assume that $a+c+2b=0$ and that $F(q, r, s, t)$ is a homogeneous function of degree $(p+2)$. Let $(z^{\ep}, w^{\ep}) \in H^2(\mathbb{R}) \times H^2(\mathbb{R})$ be the family of solutions to the scaled travelling-wave system \eqref{trav-ep} that strongly converges  to $(z_0, w_0)$ as $\epsilon \to 0^+$. Define the $\epsilon$-dependent energy norm $\|\cdot\|_{\mathcal{H}_{2, \epsilon}}$ on $H^2(\mathbb{R}) \times H^2(\mathbb{R})$ by
\[
\|(\phi, \psi)\|_{\mathcal{H}_{2. \epsilon}}^2 = \|\phi\|_{L^2}^2 + \|\psi\|_{L^2}^2 + \epsilon^{\frac{2}{p+1}}\left(\|\phi'\|_{L^2}^2 + \|\psi'\|_{L^2}^2\right) + \epsilon^{\frac{4}{p+1}}\left(\|\phi''\|_{L^2}^2 + \|\psi''\|_{L^2}^2\right).
\]
Then, there exists a constant $M > 0$, independent of $\epsilon$, such that, for all $\epsilon > 0$ sufficiently small,
\[
\epsilon^{\frac{p-3}{p+1}} \|(\partial_\epsilon w^\ep, \partial_\epsilon z^\ep)\|_{\mathcal{H}_{2,\epsilon}} \le M.
\]
\end{lem}
\begin{proof}
We begin by introducing the shorthand notation $\delta = \frac{2}{p+1}$ for the exponents, so that $\epsilon^{\frac{2}{p+1}} = \epsilon^\delta$ and $\epsilon^{\frac{4}{p+1}} = \epsilon^{2\delta}$. Differentiating the relation $\omega^2(\epsilon) = 1 - \epsilon^{2\delta}$ with respect to $\epsilon$ yields
\[
2\omega(\epsilon)\omega'(\epsilon) = -2\delta \epsilon^{2\delta-1} \implies \omega'(\epsilon) = -\frac{\delta\epsilon^{2\delta-1}}{\omega(\epsilon)}  = O\left(\epsilon^{\frac{3-p}{p+1}}\right).
\]
Now, consider the scaled travelling-wave system written in component form
\begin{equation}\label{sys:scaled}
\begin{cases}
-\omega(\epsilon)\left( w - b\epsilon^\delta w'' + b_2\epsilon^{2\delta} w^{(iv)} \right) + z + c\epsilon^\delta z'' + c_2\epsilon^{2\delta} z^{(iv)} - \ep^{\frac{4}{p+1}}G_{1}(\eta^\ep) = 0, \\
-\omega(\epsilon)\left( z - b\epsilon^\delta z'' + b_2\epsilon^{2\delta} z^{(iv)} \right) + w + a\epsilon^\delta w'' + a_2\epsilon^{2\delta} w^{(iv)} -  \ep^{\frac{4}{p+1}}G_{2}(\eta^\ep) = 0.
\end{cases}
\end{equation}
where we use (\ref{Gj-H1}) and $(\eta^\ep)^T=(z^\ep, \epsilon^{\frac{\delta}2}(z^\ep)', \epsilon^{\delta}(z^\ep)'', w^\ep, \epsilon^{\frac{\delta}2}(w^\ep)', \epsilon^{\delta}(w^\ep)'')$.
Applying the differential operator $\partial_\epsilon$ to both equations in \eqref{sys:scaled}, using the product and chain rules,  we obtain the linearized system. This system can be compactly written in matrix form as
\begin{equation}\label{sys:matrix_linearized}
\mathcal{L}_\epsilon \begin{pmatrix} \partial_\epsilon  w^\epsilon \\ \partial_\epsilon z^\epsilon \end{pmatrix} -\ep^{\frac{4}{p+1}} \nabla G(\eta^\ep)\cdot\theta^\ep= \begin{pmatrix} R_1 \\ R_2 \end{pmatrix},
\end{equation}
where $(\theta^\ep)^T=(\partial_{\ep}z^\ep, \epsilon^{\frac{\delta}2}(\partial_{\ep}z^\ep)', \epsilon^{\delta}(\partial_{\ep} z^\ep)'', \partial_{\ep}w^\ep, \epsilon^{\frac{\delta}2}(\partial_{\ep}w^\ep)', \epsilon^{\delta}(\partial_{\ep}w^\ep)'')$ and the matrix differential operator $\mathcal{L}_\epsilon = (\mathcal{L}_{\epsilon, ij})_{1 \le i,j \le 2}$ acts on the derivative vector and is explicitly defined by its components
\begin{align*}
\mathcal{L}_{\epsilon, 11} &= -\omega(\epsilon)\left(1 - b\epsilon^\delta \partial_y^2 + b_2\epsilon^{2\delta} \partial_y^4\right), \\
\mathcal{L}_{\epsilon, 12} &= \left(1 + c\epsilon^\delta \partial_y^2 + c_2\epsilon^{2\delta} \partial_y^4\right) , \\
\mathcal{L}_{\epsilon, 21} &= \left(1 + a\epsilon^\delta \partial_y^2 + a_2\epsilon^{2\delta} \partial_y^4\right), \\
\mathcal{L}_{\epsilon, 22} &= -\omega(\epsilon)\left(1 - b\epsilon^\delta \partial_y^2 + b_2\epsilon^{2\delta} \partial_y^4\right).
\end{align*}
The non-homogeneous source terms $R_1$ and $R_2$ collect all terms resulting from the explicit differentiation of the $\epsilon$-dependent coefficients and the wave speed $\omega(\epsilon)$
\begin{equation*}
\begin{split}
R_1 =& \omega'(\ep) \left( w^{\ep} - b\epsilon^\delta (w^{\ep})'' + b_2\epsilon^{2\delta} (w^{\ep})^{(iv)} \right) + \omega(\ep)  \left(-\delta b \epsilon^{\delta-1} (w^{\ep})'' + 2\delta b_2 \epsilon^{2\delta-1} (w^{\ep})^{(iv)}\right) \\&- \left(\delta c \epsilon^{\delta-1} (z^{\ep})'' + 2\delta c_2 \epsilon^{2\delta-1} (z^{\ep})^{(iv)}\right) +\ep^{\frac{3-p}{p+1}}\left(\frac{4}{p+1} G_1(\eta^\ep) + \ep \nabla G_1(\eta^\ep)\cdot\sigma^\ep\right),\\
R_2 =& \omega'(\ep)  \left( z^{\ep} - b\epsilon^\delta (z^{\ep})'' + b_2\epsilon^{2\delta} (z^{\ep})^{(iv)} \right) + \omega(\ep)  \left(-\delta b \epsilon^{\delta-1} (z^{\ep})'' + 2\delta b_2 \epsilon^{2\delta-1} (z^{\ep})^{(iv)}\right) \\&- \left(\delta a \epsilon^{\delta-1}(w^{\ep})'' + 2\delta a_2 \epsilon^{2\delta-1} (w^{\ep})^{(iv)}\right) + +\ep^{\frac{3-p}{p+1}}\left(\frac{4}{p+1} G_2(\eta^\ep) + \ep \nabla G_2(\eta^\ep)\cdot\sigma^\ep\right),
\end{split}
\end{equation*}
where $(\sigma^\ep)^T=(0, \epsilon^{\frac{\delta}2}(z^\ep)', \epsilon^{\delta}(z^\ep)'', 0, \epsilon^{\frac{\delta}2}(w^\ep)', \epsilon^{\delta}(w^\ep)'')$. 
Given that $(z^\ep, w^\ep)$ is uniformly bounded in $\mathcal H_{2, \ep}$ for small $\epsilon$, the leading order behavior of these residuals is dominated by the terms carrying the lowest power of $\epsilon$, which corresponds precisely to $\epsilon^{2\delta-1} = \epsilon^{\frac{3-p}{p+1}}$. Consequently, we have the uniform $L^2$-estimate:
\begin{equation}\label{est:R_bound}
\|R_1\|_{L^2} + \|R_2\|_{L^2} \le C_1 \epsilon^{\frac{3-p}{p+1}},
\end{equation}
where $C_1 > 0$ is independent of $\epsilon$.

To derive the energy estimate, we test the matrix system \eqref{sys:matrix_linearized} against the vector of derivatives $(\partial_\ep w^\epsilon, \partial_\ep z^\epsilon)^T$ via the $L^2(\mathbb{R}) \times L^2(\mathbb{R})$ inner product
\begin{equation}\label{eq:tested}
\left\langle \mathcal{L}_\epsilon \begin{pmatrix} \partial_\epsilon  w^\epsilon \\ \partial_\epsilon  z^\epsilon \end{pmatrix}, \begin{pmatrix} \partial_\epsilon  w^\epsilon \\\partial_\epsilon  z^\epsilon\end{pmatrix} \right\rangle_{L^2 \times L^2} = \langle R_1,\partial_\epsilon  w^\epsilon \rangle + \langle R_2, \partial_\epsilon  z^\epsilon \rangle.
\end{equation}
Expanding the left-hand side of \eqref{eq:tested}, we integrate by parts the higher-order spatial derivatives inside the diagonal operators $\mathcal{L}_{\epsilon, 11}$ and $\mathcal{L}_{\epsilon, 22}$. Exploiting the self-adjointness and coercivity of the fourth-order linear elliptic blocks yields:
\[
-\omega\langle (1 - b\epsilon^\delta \partial_y^2 + b_2\epsilon^{2\delta} \partial_y^4)\partial_\epsilon  w^\epsilon, \partial_\epsilon  w^\epsilon \rangle = -\omega \left( \|\partial_\epsilon  w^\epsilon\|_{L^2}^2 + b\epsilon^\delta \|(\partial_\epsilon  w^\epsilon)'\|_{L^2}^2 + b_2 \epsilon^{2\delta} \|(\partial_\epsilon  w^\epsilon)''\|_{L^2}^2 \right).
\]
The off-diagonal coupling terms $\mathcal{L}_{\epsilon, 12}, \mathcal{L}_{\epsilon, 21}$ and the localized linearized nonlinearity potentials containing $\nabla G(\eta^{\ep})$ contribute lower-order terms that are strictly controlled by standard Sobolev embeddings due to the homogeneity of $F$. Thus, for $\epsilon$ sufficiently small, the left-hand side of \eqref{eq:tested} matches the structure of the $\epsilon$-dependent energy norm and satisfies the coercive lower bound
\begin{equation}\label{est:coercive}
\left\langle \mathcal{L}_\epsilon \begin{pmatrix} \partial_\epsilon  w^\epsilon \\ \partial_\epsilon  z^\epsilon \end{pmatrix}-\ep^{\frac{4}{p+1}} \nabla G(\eta^\ep)\cdot\theta^\ep, \begin{pmatrix}\partial_\epsilon  w^\epsilon \\ \partial_\epsilon  z^\epsilon \end{pmatrix} \right\rangle_{L^2 \times L^2} \ge C_2 \|(\partial_\epsilon  w^\epsilon, \partial_\epsilon  z^\epsilon)\|_{\mathcal{H}_{2,\epsilon}}^2.
\end{equation}
For the right-hand side of \eqref{eq:tested}, we apply the Cauchy-Schwarz inequality followed by Young's inequality with $\gamma > 0$
\begin{align*}
\left| \langle R_1, \partial_\epsilon  w^\epsilon \rangle + \langle R_2, \partial_\epsilon  z^\epsilon\rangle \right| &\le \left( \|R_1\|_{L^2} + \|R_2\|_{L^2} \right) \|(\partial_\epsilon  w^\epsilon, \partial_\epsilon  z^\epsilon)\|_{L^2} \\
&\le \frac{1}{4\gamma}\left( \|R_1\|_{L^2} + \|R_2\|_{L^2} \right)^2 + \gamma\|(\partial_\epsilon  w^\epsilon, \partial_\epsilon  z^\epsilon)\|_{L^2}^2.
\end{align*}
Since $\|(\partial_\epsilon  w^\epsilon, \partial_\epsilon  z^\epsilon)\|_{L^2}^2 \le \|(\partial_\epsilon  w^\epsilon, \partial_\epsilon  z^\epsilon)\|_{\mathcal{H}_{2,\epsilon}}^2$, choosing $\gamma$ sufficiently small allows us to absorb the term $\gamma \|(\partial_\epsilon  w^\epsilon, \partial_\epsilon  z^\epsilon)\|_{L^2}^2$ directly into the left-hand side energy functional \eqref{est:coercive} without losing $\epsilon$-powers. Combining \eqref{est:R_bound} and \eqref{est:coercive}, we deduce
\[
\|(\partial_\epsilon  w^\epsilon, \partial_\epsilon  z^\epsilon)\|_{\mathcal{H}_{2,\epsilon}}^2 \le C_3 \left( \epsilon^{\frac{3-p}{p+1}} \right)^2.
\]
Taking the square root of both sides yields
\[
\|(\partial_\epsilon  w^\epsilon, \partial_\epsilon  z^\epsilon)\|_{\mathcal{H}_{2,\epsilon}} \le C_4 \epsilon^{\frac{3-p}{p+1}}.
\]
Finally, multiplying both sides by the target weight $\epsilon^{\frac{p-3}{p+1}}$, we obtain
\[
\epsilon^{\frac{p-3}{p+1}} \|(\partial_\epsilon  w^\epsilon, \partial_\epsilon  z^\epsilon)\|_{\mathcal{H}_{2,\epsilon}} \le C_4 \epsilon^{\frac{3-p}{p+1}} \epsilon^{\frac{p-3}{p+1}} = C_4,
\]
which is uniformly bounded as $\epsilon \to 0^+$, concluding the proof.
\end{proof}

\section{Stability}\label{sec4}

In this section, we discuss the global well-posedness briefly and derive bounds for $S$ and $S'$ to establish the positivity of $S''(\omega)$ for values of $|\omega|$ sufficiently close to $1^-$.  Before we go further, we state the definition of orbital stability of the set of ground states $\mathcal G_{\omega}$, 

\begin{defn}[Orbital stability]
We say that the ground state set $\mathcal G_{\omega}$ is orbitally stable for  $\omega\in \R^+$, if for any $\Psi_{\omega} \in \mathcal G_{\omega}$ and $\varepsilon>0$, there is $\delta(\varepsilon)>0$ such that for $U_0\in X$ with
\[
||U_0 -\Psi_{\omega}||_X< \delta, 
\]
then the initial value problem associated with the system (\ref{1bbl}) has a unique global solution $U(t)$ with $U(0)=U_0$ satisfying 
\[
\inf_{\Psi\in \mathcal G_{\omega}}||U(t)-\Psi||_X< \varepsilon.
\]
\end{defn}
\subsection{Global well-posedness} 
Although the analysis of the initial value problem associated with the Boussinesq system (\ref{1bbl}) is not part of the scope of this paper, we include a brief discussion of the global existence for small initial data, assuming local well-posedness. The latter assumption for specific nonlinearities must follow from the existence of a $C_0$ semi-group associated with the linear part of the Boussinesq system (\ref{1bbl}). Moreover, in the case where $F$ were a $(p+2)$-homogeneous function, it is possible to obtain global existence for initial data near a traveling wave solution $\Psi_{\omega}$, using the existence of invariance sets under flow for the Boussinesq system (\ref{1bbl}) in the case $S''(\omega)>0$. 

We use the Hamiltonian structure together with the extension of the existence interval for small initial data to obtain a global existence result of  \eqref{1bbl}.  In particular, from item iii) of Lemma \ref{esf-lev-2021}, given any solution of \eqref{1bbl} with $t \in [ 0, t_1]$, there are two fixed positive constants $C_0, C_1$ independent of the solution such that for any $t \in [ 0, t_1]$,
$$
C_0 \| (\eta , u )\|_{H^2(\mathbb{R})\times H^2(\mathbb{R})} ^2\left ( 1 - C_1
\left(\| (\eta , u )\|_{H^2(\mathbb{R})\times H^2(\mathbb{R})}^{q_1}+\| (\eta , u )\|_{H^2(\mathbb{R})\times H^2(\mathbb{R})}^{q_2}\right)\right ) \leq \mathcal{H} \begin{pmatrix}\eta  \\ u \end{pmatrix}= \mathcal{H} \begin{pmatrix}\eta_0 \\ u_0 \end{pmatrix}\, .
$$
Therefore, if $$\frac12 - C_1
\| (\eta , u )\|_{H^2(\mathbb{R})\times H^2(\mathbb{R})}^{q_j} \geq \frac14$$ or
$$\| (\eta , u )\|_{H^2(\mathbb{R})\times H^2(\mathbb{R})} \leq \left(\frac{1}{4C_1} \right)^{\frac{1}{q_j}},$$ for $j=1,2$,
then we get 
$$
\| (\eta , u )\|_{H^2(\mathbb{R})\times H^2(\mathbb{R})}  \leq \left ( \frac 2{C_0}\mathcal{H} \begin{pmatrix}\eta_0 \\ u_0 \end{pmatrix} \right )^{\frac12}, \qquad \mbox{with  }\quad \mathcal{H} \begin{pmatrix}\eta_0 \\ u_0 \end{pmatrix} \geq 0 \, .
$$
From the assumption on the local well-posedness above, we choose $\delta_0> 0$ with $2\delta_0  \leq \left(\frac{1}{4C_1} \right)^{\frac{1}{q_j}}$ for $j=1,2$ and let the initial condition satisfy
$$
\| (\eta_0 , u _0)\|_{H^2(\mathbb{R})\times H^2(\mathbb{R})} \leq \delta_0
\quad \mbox{and}\quad \left ( \frac 2{C_0}\mathcal{H} \begin{pmatrix}\eta_0 \\ u_0 \end{pmatrix} \right )^{1/2} \leq \delta_0\, .
$$
Then, the solution of \eqref{1bbl} exists for $t\in [0, T_0]$ and satisfies
$
\| (\eta , u )\|_{H^2(\mathbb{R})\times H^2(\mathbb{R})}  \leq \delta_0\,
$ for all $t \in [0, T_0]$. Hence, from the local well-posedness result again, we can extend the solution to $[0, 2T_0]$. Continuing this procedure produces a global solution for $t \in [0, \infty)$, which is unique and satisfies $
\| (\eta , u )\|_{H^2(\mathbb{R})\times H^2(\mathbb{R})}  \leq \delta_0\,
$ for all $t \in [0, \infty)$.
\begin{rem}\label{G-e}
We want to point out that the estimate (\ref{est1}) implies that when $\omega \to 1^{-}$, $(\psi_{\omega}, v_{\omega})$ is small in $H^2(\R)\times H^2(\R)$. Therefore, if the initial condition of (\ref{1bbl}) is near $(\psi_{\omega}, v_{\omega})\in \mathcal G_{\omega}$ for $\omega$ near $1^{-}$, then the global existence result implies that the solution of (\ref{1bbl}) exists for $t\in [0, \infty)$. 
\end{rem}
Now, we establish the existence of a map that associates a speed
$\omega$ to each $V$ in some neighborhood of $\mathcal G_{\omega}$. For a nonempty set $S\subset (H^2(\R^2))^2$, we define 
\[
\mathcal V_{\varepsilon}(S)= \cup_{V\in S}B_{\varepsilon}(V)=\{U\in (H^2(\R^2))^2: ||V-U||_X<\varepsilon \ \mbox{for some $V\in S$}\}.
\]
\begin{lem} There exist $\varepsilon>0$ and a continuous map $\omega: \mathcal V_{\varepsilon}(\mathcal G_{\omega})\to \R$
such that $$S(\omega(U))= J_1(U)=\frac12 N(U)-K(U),$$ for each $U \in  \mathcal V_{\varepsilon}(\mathcal G_{\omega})$.
\end{lem}  
\begin{proof}
Recall that $S$ is a continuous and strictly decreasing function. Then, it follows that $S^{-1}$ is a continuous and strictly decreasing function on the range of $S$. Now, since for any $\Psi\in G_{\omega}$ \ $(I_{2, \omega}(\psi)=N(\psi))$
\[
S(\omega) = J_{\omega}(\Psi) = J_1(\Psi), 
\]
then,  from the continuity of $S_1$ we obtain that 
there exists $\varepsilon>0$  such that $S_1(U)$ is in the range of $J_{\omega}$ for any $U\in G_{\omega}$. We may  define
\begin{equation}
\omega(U)= S^{-1}(J_1(U))
\end{equation}
for any such $U$. The continuity of this map follows from the continuity of $S^{-1}$ and $J_1$.
\end{proof}
The key result needed in the stability analysis is the following:
\begin{lem}\label{prelem} Suppose that $S''(\omega)>0$. Then there exists $\varepsilon>0$ such that for $\Psi
\in {\mathcal{G}}_{\omega}$ and $V\in \mathcal V_{\varepsilon}(\mathcal G_{\omega})$ it follows that
$$
\mathcal H(V)-\mathcal H\left(\Psi\right)+\omega(V)\left(\mathcal{Q}(V)-\mathcal{Q}\left(\Psi\right)\right) \geq \frac{1}{4} S^{\prime \prime}(\omega)|\omega(V)-\omega|^2.
$$
\end{lem}
\begin{proof} Using that $S'(\omega)=\mathcal Q(\Psi)$  and Taylor expansion, 
$$
S(\omega_1)\geq S(\omega)+S'(\omega)(\omega_1-\omega)+ \frac{1}{4} S^{\prime \prime}(\omega)|\omega_1-\omega|^2.
$$
for $\omega_1$ sufficiently close to $\omega$, which implies that
$$
S(\omega(V))\geq  H\left(\Psi\right)+\mathcal Q(\Psi)\omega(V)+ \frac{1}{4} S^{\prime \prime}(\omega)|\omega(V)-\omega|^2.
$$
for $V\in \mathcal V_{\varepsilon}(\mathcal G_{\omega})$ and $\varepsilon>0$ but small enough. If we had $P_{\omega(V)}(V)\leq 0$, we chose $\Phi\in {\mathcal{G}}_{\omega(V)}$ then, 
\begin{align*}
S(\omega(V))&=J_{\omega(V), p}(\Phi)
\leq J_{\omega(V), p}(V)= J_{\omega(V)}(V) -\frac1{p+2} P_{\omega(V)}(V)\\
&\leq J_{\omega(V)}(V) -\frac1{2} P_{\omega(V)}(V)
= J_{1}(V) = S(\omega(V)). 
\end{align*}
On the other hand, If we had $P_{\omega(V)}(V)\geq 0$, then
\[
S(\omega(V))=J_{1}(V)=J_{\omega(V)}(V) -\frac1{p+2} P_{\omega(V)}(V)\leq J_{\omega(V)}(V)= \mathcal H(V)+\omega(V)  \mathcal Q(V).
\]
So, we find that all of the quantities above are equivalent, meaning that $V$ achieves the same minimum as $\Phi$, and so $V\in \mathcal G_{\omega(V)}$. In other words, 
\[
S(\omega(V))= J_{\omega(V)}(V)=\mathcal H(V)+\omega(V)  \mathcal Q(V).
\]
Using this fact in a previous inequality, we get the desired result.
\end{proof}
\begin{thm}\label{t-45} Let $0<\omega<\min\left\{1,\frac{-a}{b},\frac{-c}{b}, \frac{a_2}{b_2},\frac{c_2}{b_2}\right\}$ and the wave speed $0<{\omega}<1$ near $1$ with $S''(\omega)>0$. Then the set of ground states $\mathcal G_{\omega}$ is stable.
\end{thm}
\begin{proof}
We argue by contradiction. If $S''(\omega) >0$ and the ground state set $\mathcal G_{\omega}$ is unstable, Then for $\Psi_{\omega}\in \mathcal G_{\omega}$, there exists $\varepsilon_0>0$ such that for $0<\varepsilon<\varepsilon_0$ 
there exists a sequence of initial data $(U_0^k)_k \subset \mathcal V_{\frac1k}(\mathcal G_{\omega})$, and $t_k > 0$ such that the solution $U_k(t)$ of the Boussinesq system (\ref{1bbl}) with initial data $U(0)=U_0^k$ satisfies (see Remark  \ref{G-e})
\begin{equation}\label{contradiction}
\lim_{k\to \infty}||U_0^k-\Psi_{\omega}||_X =0 \ \ \mbox{and} \ \ \inf_{V\in\mathcal G_{{\omega}}}\|U_k(t_k)-V\|_{X}\geq \varepsilon>0.
\end{equation}
Now, using Lemma \ref{prelem},  the invariance of $\mathcal G_{\omega}$ and $\mathcal Q$ on solutions and the continuity of $S$, we conclude that 
\[
\lim_{k \to \infty}J_{1}(U_k(t_k))=\lim_{k \to \infty}S(\omega(U_k(t_k)))=S(\omega).
\]
Moreover, we also have 
\[
\lim_{k \to \infty}J_{\omega}(U_k(t_k))=\lim_{k \to \infty}(\mathcal H_{\omega}(U_k(t_k))+\omega \mathcal Q(U_k(t_k)))=\lim_{k \to \infty}(\mathcal H_{\omega}(U_k(0))+\omega \mathcal Q(U_k(0)))=S(\omega).
\]

On the other hand, we have
\[
\lim_{k \to \infty}\frac12 P_{\omega}(U_k(t_k))=\lim_{k \to \infty}( J_{\omega}(U_k(t_k))-J_{1}(U_k(t_k)))=0, 
\]
which means that $U_k(t_k)$ is a minimizing sequence. Then, there is a subsequence $U_{k_j}(t_{k_j})$ and $\Psi_j \in \mathcal G_{\omega}$ such 
$$
\lim_{j\rightarrow\infty}\|U_{k_j}(t_{k_j})- \Psi_j\|_{X}=0.
$$
This clearly contradicts the condition (\ref{contradiction}), and thus the set of ground state solutions $\mathcal G_{\omega}$ is stable.
\end{proof}
Now, we analyze the asymptotic behavior of $S$ and $S'$ as $|\omega|\to1^-$, in order to determine the sign of  $S''(\omega)$ when $\omega$ is near $1^{-}$. Recall from Lemma \ref{estd} that
\begin{equation}\label{deri}
S'(\omega)=I_{2}(u^{\omega},v^{\omega})=\mathcal Q(u^{\omega},v^{\omega}).
\end{equation}
\begin{thm}\label{behd}
Let $0<|\omega|<\min\left\{1,\frac{-a}{b},\frac{-c}{b},
\frac{a_2}{b_2},\frac{c_2}{b_2}\right\}$,
and $(u^{\omega},v^{\omega})\in\mathcal{G}_{\omega}$. Then
\[
\lim_{\omega\to1^-}S(\omega)=0,\quad  \text{and} \quad
I_{2,\omega}(u^{\omega},v^{\omega})<0,
\]
for $\omega$ sufficiently close to $1^-$.
\end{thm}
\begin{proof}
From the scaling relation established in \eqref{Ic=eIe}, together with Lemma \ref{lim-2}, we have
\[
S(\omega)=\epsilon^{\frac{3p+8}{p(p+1)}}S^\epsilon,
\]
where $\epsilon\to0^+$ as $\omega\to1^-$. Since
\[
\lim_{\epsilon\to0^+}S^\epsilon=\mathcal J_0<\infty,
\]
it follows immediately that
\[
\lim_{\omega\to1^-}S(\omega)=0.
\]
Next, using the same notation introduced in Section \ref{sec3}, we write
\begin{align*}
\epsilon^{-\frac{4}{p+1}}
I_{2}^{\epsilon}(z^\epsilon,w^\epsilon)
=&
-2\omega(\epsilon)
\int_{\mathbb R}
\Big(
z^\epsilon w^\epsilon
+
b\epsilon^{\frac{2}{p+1}}
z_x^\epsilon w_x^\epsilon
+
b_2\epsilon^{\frac{4}{p+1}}
z_{xx}^\epsilon w_{xx}^\epsilon
\Big)\,dx.
\end{align*}
From Lemma \ref{conver}, after translation if necessary,
\[
(z^\epsilon,w^\epsilon)\to(w_0,w_0)
\quad \text{strongly in } H^2(\mathbb R)\times H^2(\mathbb R),
\]
with $w_0\neq0$. Therefore,
\begin{align*}
\lim_{\epsilon\to0^+}
\epsilon^{-\frac{4}{p+1}}
I_{2}^{\epsilon}(z^\epsilon,w^\epsilon)
&=
-2\int_{\mathbb R}w_0^2\,dx
<0.
\end{align*}
Consequently,
\[
I_{2}^{\epsilon}(z^\epsilon,w^\epsilon)<0,
\]
for $\epsilon>0$ sufficiently small. By the correspondence between the renormalized variables and the original traveling-wave family, this yields
\[
I_{2,\omega}(u^\omega,v^\omega)<0,
\]
for $\omega$ sufficiently close to $1^-$.
\end{proof}

\begin{thm}
There exists $\omega_0\in(0,1)$ sufficiently close to $1$ such that $S$ is decreasing on $(\omega_0,1)$. Moreover,
\[
\lim_{\omega\to1^-}S'(\omega)=0.
\]
\end{thm}

\begin{proof}
From Theorem \ref{behd} and identity \eqref{deri}, we have
\[
S'(\omega)=I_{2}(u^\omega,v^\omega)<0,
\]
for $\omega$ sufficiently close to $1^-$. Hence, $S$ is decreasing in a neighborhood of $1$.

On the other hand, from Lemma \ref{estd}-(ii) and Lemma \ref{esf-lev-2021}-(vii), there exist positive constants $C_j$ such that
\begin{align*}
I_{1}(u^\omega,v^\omega)
&\leq
C_1\|(u^\omega,v^\omega)\|_X^2
\leq
C_2 I_\omega(u^\omega,v^\omega)
\leq
C_3 S(\omega),
\end{align*}
and
\begin{align*}
|K(u^\omega,v^\omega)|
\leq
C_4\Big(
S(\omega)^{\frac{q_1+2}{2}}
+
S(\omega)^{\frac{q_2+2}{2}}
\Big).
\end{align*}
Using the definition of $J_\omega$, we obtain
\[
S(\omega)
= \frac12 I_1(u^\omega,v^\omega)+\omega I_{2}(u^\omega,v^\omega)
-
K(u^\omega,v^\omega).
\]
Combining this identity with \eqref{deri}, we deduce
\[
\omega S'(\omega)
=
S(\omega)
-\frac12 I_1(u^\omega,v^\omega)
+
K(u^\omega,v^\omega).
\]
Since
\[
\lim_{\omega\to1^-}S(\omega)=0,
\]
the estimates above imply
\[
\lim_{\omega\to1^-}I_1(u^\omega,v^\omega)=0,
\qquad
\lim_{\omega\to1^-}K(u^\omega,v^\omega)=0.
\]
Therefore,
\[
\lim_{\omega\to1^-}S'(\omega)=0.
\]
This completes the proof.
\end{proof}

\subsection{Properties of the scalar function in the homogeneous case} In this subsection, we state some properties of the function $S$,  if $F(q, r, s, t)$ is a homogeneous function of degree $(p+2)$. 
In this case, we use the fact that $N(u, v)=(p+2)K(u, v)$, which allows us to define
$$
    \mathcal I_{\omega}= \inf_{\Psi\in X}\{I_{\omega}(\Psi): K(\Psi)=1\}. 
$$
As shown in \cite{CQS-2024}, we have
 \begin{equation}\label{s_w}
S(\omega)
= \frac{p}{2(p+2)}\left(\frac{1}{p+2}\right)^{\frac{2}{p}}\left ( I_\omega(\Psi_{\omega})\right ) ^{\frac{p+2}p},
\end{equation}
which is uniquely defined  for any $\Psi_{\omega} \in \mathcal{G}_\omega$. Moreover, for $\Psi_{\omega}\in\mathcal{G}_{\omega}$,  it follows that
\begin{equation}\label{S-pr}
S\, '({\omega})=\frac{1}{2}\left(\frac{1}{p+2}\right)^{\frac{2}{p}}\frac{I_{2,{\omega}}(\Psi_{\omega})}{\omega}(I_{\omega}(\Psi_{\omega}))^{\frac{2}{p}}= \mathcal Q(\widetilde{\Psi}_{\omega})\,,
\end{equation}
where $\widetilde{\Psi}_{\omega}=\left(\frac{\mathcal I_{\omega}}{p+2}\right)^{\frac{1}{p}} {\Psi}_{\omega}$ corresponds to a minimizer for $S(\omega)$.
Additionally, we have $S\, '(\omega)<0$ for $\omega$ near $1^{-}$.

\color{red}

\color{black}

Using the previous results, we compute $S''(\omega)$ for  $\omega$  near $1^{-}$. 

\begin{lem}\label{lem48}
Let $F(q, r, s, t)$ be a homogeneous function of degree $(p+2)$.  For $\Psi_{\omega}\in \mathcal G_{\omega}$, we have
\begin{equation*}
S''({\omega})=\frac{1}{p}\left(\frac{1}{p+2}\right)^{\frac{2}{p}}\left(\frac{I_{2,{\omega}}(\Psi_{\omega})}{\omega}\right)^{2}I_{\omega}(\Psi_{\omega})^{\frac{2-p}{p}}+\frac{1}{2}\left(\frac{1}{p+2}\right)^{\frac{2}{p}}I_{\omega}(\Psi_{\omega})^{\frac{2}{p}}\frac{d}{d\omega}\left(\frac{I_{2,{\omega}}(\Psi_{\omega})}{\omega}\right).
\end{equation*}
Moreover, there is $\omega*\in (0, 1)$ such that  $S''(\omega)>0$, for $\omega\in (\omega^*, 1)$ and $ 0<p <8$.  In the case where $F(\cdot, r, \cdot, t)$ is a homogeneous function of degree $\beta$, there is $\omega*\in (0, 1)$ such that for $\omega\in (\omega^*, 1)$ and $ 0<p+2\beta <8$, $S''(\omega)>0$.
\end{lem}
\begin{proof} Differentiating with respect to $\omega$ the formula for $S$ given by  (\ref{s_w}) and comparing with the formula for $S'$ given by (\ref{S-pr}), we get
\[
\frac{d}{d\omega} I_{\omega}(\Psi_{\omega})= \frac{I_{2, \omega}(\Psi_{\omega})}{\omega}
\]
Then, differentiating with respect to $\omega$ the formula for $S'$ given (\ref{S-pr}) and using the previous fact, we get the desired formula for $S''$. 

Now, we point out that the first term of $S''$ is nonnegative, so we will determine the sign of
\[
\frac d{d\omega}
\left(
\frac{I_{2,\omega}(\Psi_\omega)}{\omega}
\right).
\]
Using the KdV scaling induced by (\ref{kdv-s}) and definition (\ref{Ic=eIe}), we obtain
\[
I_2^\epsilon(z,w)
=-2\omega
\int_{\mathbb R}
\Big(
\epsilon^{-\frac4{p+1}}zw
+
b\epsilon^{-\frac2{p+1}}z_yw_y
+
b_2 z_{yy}w_{yy}
\Big)\,dy.
\]
Therefore,
\[
\frac{I_{2,\omega}(\Psi_\omega)}{\omega}=\frac{\epsilon^{\frac{3p+8}{p(p+1)}}}{\omega}
I_2^\epsilon(z, w)
=
-2
\epsilon^{\frac{8-p}{p(p+1)}}A(\epsilon),
\]
where $A(\epsilon)$ is given by
\[
A(\epsilon)
=
\int_{\mathbb R}
\Big(
zw
+
b\epsilon^{\frac2{p+1}}
z_yw_y
+
b_2
\epsilon^{\frac4{p+1}}
z_{yy}w_{yy}
\Big)\,dy.
\]
From the previous notation, we have 
\begin{equation}\label{eq:key2}
\begin{split}
\frac d{d\omega}
\left(
\frac{I_{2,\omega}(\Psi_\omega)}{\omega}
\right)&=\frac d{d\epsilon}
\left(
\frac{\epsilon^{\frac{3p+8}{p(p+1)}}}{\omega(\ep)}
I_2^\epsilon(z, w)
\right)\frac{d \ep}{d\omega}\\
&=
(p+1)\omega(\ep)\epsilon^{\frac{p-3}{p+1}}\left(\frac{(8-p)}{p(p+1)}
\epsilon^{-\frac{(p+4)(p-2)}{p(p+1)}}
A(\epsilon)
+\epsilon^{\frac{8-p}{p(p+1)}}
A'(\epsilon)\right)\\
&=
\frac{(8-p)}{p}\omega(\ep)\epsilon^{\frac{8-p}{p(p+1)}}
A(\epsilon)
+(p+1)\omega(\ep)\epsilon^{\frac{8-p}{p(p+1)}}
\epsilon^{\frac{p-3}{p+1}}A'(\epsilon),
\end{split}
\end{equation}
where we are using $\omega^{2}=1-\epsilon^{\frac{4}{p+1}}$. 

On the other hand, differentiating the functional \(A(\epsilon)\) yields terms involving
\(\partial_\epsilon z\), \(\partial_\epsilon w\), as well as derivatives of the coefficients
\(\epsilon^{\frac{2}{p+1}}\) and \(\epsilon^{\frac{4}{p+1}}\). Hence $A'(\epsilon) =
R(z,w,\partial_\epsilon z,\partial_\epsilon w)$, where \(R\) is linear with respect to either \(\partial_\epsilon z\) or \(\partial_\epsilon w\), and whose coefficients remain uniformly bounded for sufficiently small
\(\epsilon>0\).  Now, from Lemma \ref{lem:H2-bound}, we have that $\epsilon^{\frac{p-3}{p+1}}
(\partial_\epsilon z,\partial_\epsilon w)$ is uniformly bounded in $\mathcal H_{2, \ep},$
for \(0<p<8\). In particular, this bounded weighted energy norm yields uniform control over the lower-order regularized components, implying that the residual term satisfies
\[
R(z,w,\partial_\epsilon z,\partial_\epsilon w)
=
O\!\left(
\epsilon^{-\frac{p-3}{p+1}}
\right),
\]
and therefore
\[
\epsilon^{\frac{p-3}{p+1}}
\epsilon^{\frac{8-p}{p(p+1)}}
R(z,w,\partial_\epsilon z,\partial_\epsilon w)
=
O\!\left(
\epsilon^{\frac{8-p}{p(p+1)}}
\right),
\]
which converges to zero as \(\epsilon\to0^+\) whenever \(0<p<8\). It follows from \eqref{eq:key2} that
\begin{equation}\label{eq:key3}
\frac d{d\omega}
\left(
\frac{I_{2,\omega}(\Psi_\omega)}{\omega}
\right)
=
\frac{(8-p)}{p}\omega(\ep)\epsilon^{\frac{8-p}{p(p+1)}}
A(\epsilon)
+o
\left(1\right).
\end{equation}
Thanks to the Lemma \ref{conver},
\[
(z^\epsilon,w^\epsilon)
\to
(w_0,w_0)
\quad
\text{strongly in }
H^2(\mathbb R)\times H^2(\mathbb R).
\]
Consequently,
\[
A(\epsilon)
=
\int_{\mathbb R}
w_0^2\,dy
+
o(1).
\]
Substituting into \eqref{eq:key3} yields
\begin{equation}\label{eq:key4}
\frac d{d\omega}
\left(
\frac{I_{2,\omega}(\Psi_\omega)}{\omega}
\right)
=
\frac{(8-p)}{p}\epsilon^{\frac{8-p}{p(p+1)}}
\left(
\int_{\mathbb R}w_0^2\,dy
\right)
+o
\left(1
\right),
\end{equation}
which implies for $0<p<8$ that the left side of 
\eqref{eq:key4} is strictly positive.
Hence
\[
\frac d{d\omega}
\left(
\frac{I_{2,\omega}}{\omega}
\right)
>0,
\]
for $\epsilon$ sufficiently small. Returning to the formula for $S''$, we conclude that both terms
on the right-hand side are positive whenever $\omega$
is sufficiently close to $1^{-}$.
Therefore, there is $0<\omega^*<1$ such that
\[
S''(\omega)>0,
\]
for $\omega\in(\omega^*, 1)$ and $0<p<8$, as desired.


In the case where $F(\cdot, r, \cdot, t)$ is a homogeneous function of degree $\beta$, we need to modify the definition of $\alpha= \ep^{\frac{4}{p(p+1)}}$ for $\alpha= \ep^{\frac{4-\beta}{p(p+1)}}$ and the functionals involved in formula (\ref{Ic=eIe}) by 
\[
L(\Psi_{\omega})= \ep^{\frac{3p+8-2\beta}{p(p+1)}}L^{\ep}(\Psi^{\ep}),
\]
where $L$ stands for $J_{\omega}$, $I_{\omega}$, $I_{1}$, $I_{2, \omega}$  and $K$, and $L^\ep$ stands for $J^\ep$, $I^\ep$, $I^{1,\ep}$, $I^{2, \ep}$ and $K^\ep$.  As above, using the modified KdV scaling induced by (\ref{kdv-s}) and definition (\ref{Ic=eIe}), we obtain in this case that
\[
\frac{\epsilon^{\frac{3p+8-2\beta}{p(p+1)}}}{\omega(\ep)}
I_2^\epsilon(z, w)
=
-2
\epsilon^{\frac{8-(p+2\beta)}{p(p+1)}}A(\epsilon), 
\]
and also that 
\begin{equation*}
\frac d{d\omega}
\left(
\frac{I_{2,\omega}(\Psi_\omega)}{\omega}
\right)
=
\frac{(8-(p+2\beta))}{p}\omega(\ep)\epsilon^{\frac{8-(p+2\beta)}{p(p+1)}}
A(\epsilon)
+o
\left(1
\right).
\end{equation*}
From the same argument in the previous case, there is $0<\omega^*<1$ such that
\[
S''(\omega)>0,
\]
for $\omega\in(\omega^*, 1)$ and $p+2\beta<8$,  as desired.
\end{proof}

\color{black}

From the previous results, we are able to complete the proof of the stability result in Theorem \ref{main}. 

\begin{proof}[{\bf Proof of Theorem \ref{main}}] 
\nd a) First, we know that  $S$ is a decreasing function and that
\[
\lim_{\omega\rightarrow1^-}S({\omega})=\lim_{\omega\rightarrow1^-}S'({\omega})=0.
\]
On the other hand, for  (\ref{est1}) we have $S(\omega)= O\left((1-\omega^2)^{\frac{3p+8}{4p}}\right)$ for $|\omega|$ near $1^{-}$, since 
\[ 
\frac{\mathcal J_0 }{2}(1-\omega^2)^{\frac{3p+8}{4p}}\leq S(\omega)=(1-\omega^2)^{\frac{3p+8}{4p}} S^{\ep} \leq \frac{3\mathcal J_0 }{2}(1-\omega^2)^{\frac{3p+8}{4p}}. 
\] 
Now, from Lemma \ref{estd} we know that 
\begin{align*}
S'(\omega)&= I_2(u^{\omega},v^{\omega})\\
&=-
2 \omega (1-\omega^2)^{\frac{8-p}{p(p+1)}}\int_{\mathbb{R}}(z^{\ep} w^{\ep}+b\epsilon^{\frac{2}{p+1}}(z^{\ep})' (w^{\ep})'+b_2 \epsilon^{\frac{4}{p+1}}(z^{\ep})'' (w^{\ep})'')\,dy\\
&= O\left( (1-\omega^2)^{\frac{8-p}{p(p+1)}}\right),
\end{align*}
for $\omega$ near $1^{-}$, where we use the convergence property of the sequence $(z^{\ep}, w^{\ep})_\ep$ obtained in Lemma \ref{lim-1}. From the fact that   $S'$ is a $C^1$ negative function for $\omega$ near $1^{-}$, there must exist intervals arbitrarily close to $1^{-}$ in which $S''$  is positive, which proves a). 

\medskip
\nd b) Assume that $F$ is $(p+2)$-homogeneous with $0<p<8$. 

Let $(\psi_\omega,v_\omega)\in \mathcal G_\omega$ be a ground state solution associated with
\[
S(\omega)
=
\inf\left\{
J_\omega(\psi,v):
(\psi,v)\in X\setminus\{0\},
\quad
P_\omega(\psi,v)=0
\right\}.
\]
From Theorem \ref{behd} and the subsequent asymptotic analysis, we know that
\[
\lim_{\omega\to1^-}S(\omega)=0,
\qquad
\lim_{\omega\to1^-}S'(\omega)=0.
\]
Moreover, from Lemma \ref{lem48}, there exists $\omega^\ast\in(0,1)$ such that for $\omega^*< \omega<1$ and $0<p<8$ we have
\[
S''(\omega)>0. 
\]
On the other hand, from Theorem \ref{t-45}, the ground state set $\mathcal G_{\omega}$ is stable for $\omega^*< \omega<1$ and $0<p<8$, giving b). 

\medskip
\nd c) Assume that $F$ is $(p+2)$-homogeneous and $F(\cdot, r, \cdot, t)$ is a homogeneous function of degree $\beta$ such that $0<p+2\beta<8$. From the discussion in literal b) and Theorem \ref{t-45}, the ground state set $\mathcal G_{\omega}$ is stable for $\omega^*< \omega<1$, proving c), and the theorem is shown. 
\end{proof}



\section{Numerical experiments}\label{sec5}

In this section, we numerically investigate the stability of the family of solitary wave solutions discussed previously, focusing on the wave-velocity regime $0 < \omega < 1$ and modeling parameters satisfying the condition $a + c + 2b = 0$.

To compute solitary wave solutions of the system \eqref{1bbl} for arbitrary wave velocities, we adapt the numerical scheme introduced in \cite{Esfahani-Levandosky-2021}, originally developed for the fifth-order scalar KdV equation. This method is modified to approximate solutions of the system of traveling-wave equations \eqref{trav-eqs}, employing a Fourier spectral approach 
augmented with stabilizing factors for both the wave elevation $\eta$ and the fluid velocity $u$, thereby ensuring convergence of the iterative scheme used in our simulations.

The function $S'(\omega)$ is computed using the analytical expression $S'(\omega) = I_2(\Psi_{\omega})$, obtained in Lemma \ref{estd}, while $S''(\omega)$ is approximated using a second-order central difference scheme applied to $S'(\omega)$. As established in the analytical results presented in the previous section, a traveling-wave solution corresponding to the wave velocity $\omega$ is stable when $S''(\omega)>0$, and unstable when $S''(\omega) < 0$.

To derive the numerical scheme, we apply the Fourier transform to equations \eqref{trav-eqs}, yielding
\begin{align*}
\begin{cases}
-\omega ( \hat{v}_k + d k^2 \hat{v}_k + d_2 k^4 \hat{v}_k ) + \hat{\psi}_k - c k^2 \hat{\psi}_k + c_2 k^4 \hat{\psi}_k &=\widehat{G_1}, \\
-\omega( \hat{\psi}_k + b k^2 \hat{\psi}_k + b_2 k^4 \hat{\psi}_k ) + \hat{v}_k - a k^2 \hat{v}_k + a_2 k^4 \hat{v}_k & = \widehat{G_2}.
\end{cases}
\end{align*}
We recall that $b=d>0$, $b_2 = d_2 >0$, $a, c <0$, $a_2, c_2 >0$, and $\widehat{G_1}, \widehat{G_2}$ denote the Fourier transforms of the nonlinear functions $G_1(\psi, \psi', \psi'', v, v', v'')$ and $G_2(\psi, \psi', \psi'', v, v', v'')$, respectively.

Rewriting these equations, we get
\begin{align*}
\begin{cases}
-\omega (1 + d k^2 + d_2 k^4) \hat{v}_k + (1-ck^2 + c_2 k^4) \hat{\psi}_k &= \widehat{G_1}, \\
(1-ak^2 + a_2 k^4)\hat{v}_k - \omega ( 1 + b k^2 + b_2 k^4) \hat{\psi}_k &= \widehat{G_2},
\end{cases}
\end{align*}
and thus, we can write
\begin{align}
&\hat{v}_k = \frac{ \widehat{G_1} D_{22} - \widehat{G_2}  D_{12} }{D_{11} D_{22} - D_{21} D_{12} }, \label{v_eq}
\end{align}
and
\begin{align}
&\hat{\psi}_k = \frac{ \widehat{G_2} D_{11} - \widehat{G_1} D_{21} }{ D_{11} D_{22} - D_{21} D_{12} }, \label{psi_eq}
\end{align}
where
\begin{equation*}
D = \begin{pmatrix}
D_{11} & D_{12} \\
D_{21} & D_{22} 
\end{pmatrix} =
\begin{pmatrix}
-\omega ( 1 + d k^2 + d_2 k^4 ) & 1 - c k^2 + c_2 k^4 \\
1 - a k^2 + a_2 k^4 & -\omega (1 + b k^2 + b_2 k^4 )
\end{pmatrix}.
\end{equation*}

On the other hand, by multiplying equation \eqref{v_eq} by $\hat{v}_k$ and equation \eqref{psi_eq} by $\hat{\psi}_k$, and then adding over $k$, we obtain
\begin{align*}
&P_1(v ) := \frac{ \sum_k  (D_{11} D_{22} - D_{21} D_{12} ) \hat{v}_k^2  }{ \sum_k (\widehat{G}_1 D_{22} - \widehat{G_2} D_{12}  ) \hat{v}_k } = 1,
\end{align*}
and
\begin{align*}
&P_2(\psi) := \frac{  \sum_k ( D_{11} D_{22} - D_{21} D_{12}  ) \hat{\psi}_k^2  }{  \sum_k ( \widehat{G_2} D_{11} - \widehat{G_1} D_{21} ) \hat{\psi}_k } = 1.
\end{align*}
To compute approximate solutions to the traveling wave equations, we select initial values $\psi^0$ and $v^0$. Then, for $s \geq 0$, we define the sequence
\begin{align}
&\hat{v}_{k}^{s+1} = \frac{\widehat{G_1}(\tilde{\psi^s}, \tilde{\psi^s}', \tilde{\psi^s}'', \tilde{v^s}, \tilde{v^s}', \tilde{v^s}'') D_{22} - \widehat{G_2}(\tilde{\psi^s}, \tilde{\psi^s}', \tilde{\psi^s}'', \tilde{v^s}, \tilde{v^s}', \tilde{v^s}'')  D_{12}  }{ D_{11} D_{22} - D_{21} D_{12} }, \label{iteration_1} 
\end{align}
and
\begin{align}
&\hat{\psi}_k^{s+1} =  \frac{\widehat{G_2}(\tilde{\psi^s}, \tilde{\psi^s}', \tilde{\psi^s}'', \tilde{v^s}, \tilde{v^s}', \tilde{v^s}'') D_{11} - \widehat{G_1}(\tilde{\psi^s}, \tilde{\psi^s}', \tilde{\psi^s}'', \tilde{v^s}, \tilde{v^s}', \tilde{v^s}'')  D_{21}  }{ D_{11} D_{22} - D_{21} D_{12} }, \label{iteration_2}
\end{align}
where $\tilde{v}^s : = \alpha_s v^s$, $\tilde{\psi}^s := \beta_s \psi^s$, and $\alpha_s$, $\beta_s \in \mathbb{R}$ are solutions of the following equations
\begin{equation}\label{Pequations}
\begin{cases}
P_1(\alpha_s v^s ) = 1,\\
P_2(\beta_s \psi^s ) = 1.
\end{cases}
\end{equation}
We note that the quantities $\alpha_s$ and $\beta_s$ are stabilizing factors introduced to ensure the convergence of the iteration defined in \eqref{iteration_1} and \eqref{iteration_2}. 
In the numerical simulations, we initiate the iteration \eqref{iteration_1}-\eqref{iteration_2} with 
\[
v^0(x) = \psi^0(x) = e^{-0.05 (x-50)^2},
\]
and we take $N=2^{12}$ points along the computational domain is $[0,L]=[0,100]$.

\subsection{Homogeneous case}

Suppose that the functions $G_1, G_2$ are both homogeneous of order $p+1$. Then, given that the parameters $\alpha_s$, $\beta_s$ are chosen at each iteration $s$ such that $P_1(\alpha_s v^s) = 1$ and 
$P_2(\beta_s \psi^s) = 1$, we obtain that
\begin{equation}\label{solit_curve}
\begin{cases}
&\frac{ \sum_k  (D_{11} D_{22} - D_{21} D_{12} ) \alpha_s^2 \hat{v}_k^2  }{ \sum_k \alpha_s^{p+1} (\widehat{G}_1 D_{22} - \widehat{G_2} D_{12}  ) \alpha_s \hat{v}_k } = 1,\\
\\
&\frac{  \sum_k ( D_{11} D_{22} - D_{21} D_{12}  ) \beta_s^2 \hat{\psi}_k^2  }{  \sum_k \beta_s^{p+1} ( \widehat{G_2} D_{11} - \widehat{G_1} D_{21} ) \beta_s \hat{\psi}_k } = 1.
\end{cases}
\end{equation}
From equations \eqref{Pequations} and \eqref{solit_curve}, we obtain the explicit expressions for the parameters $\alpha_s, \beta_s$ given by 
\begin{equation}\label{stabilizing1}
\begin{cases}
\alpha_s &= M_s^{1/p},\\
\beta_s &= N_s^{1/p},
\end{cases}
\end{equation}
where
\begin{align*}
&M_s := \frac{ \sum_k  (D_{11} D_{22} - D_{21} D_{12} ) \hat{v}_k^2  }{ \sum_k  (\widehat{G}_1 D_{22} - \widehat{G_2} D_{12}  ) \hat{v}_k },
\end{align*}
and
\begin{align*}
&N_s := \frac{  \sum_k ( D_{11} D_{22} - D_{21} D_{12}  ) \hat{\psi}_k^2  }{  \sum_k ( \widehat{G_2} D_{11} - \widehat{G_1} D_{21} ) \hat{\psi}_k }.
\end{align*}

As a consequence, the iteration defined by equations \eqref{iteration_1} and \eqref{iteration_2}, can be rewritten as follows 
\begin{align*}
&\hat{v}_{k}^{s+1} = M_s^{ \frac{p+1}{p} } \frac{\widehat{G_1} (U) D_{22} - \widehat{G_2}(U)  D_{12}  }{ D_{11} D_{22} - D_{21} D_{12} }, 
\end{align*}
and
\begin{align*}
&\hat{\psi}_k^{s+1} = N_s^{\frac{p+1}{p} } \frac{\widehat{G_2}(U) D_{11} - \widehat{G_1}(U)  D_{21}  }{ D_{11} D_{22} - D_{21} D_{12} }, 
\end{align*}
where $U = ( {\psi^s}, {\psi^{s}}', {\psi^{s}}'', {v^s}, {v^{s}}', {v^{s}}'')$.

We include some plots of the function $S''(\omega)$ whose sign decides the stability/instability of a traveling-wave solution of system \eqref{1bbl}. 

First, we consider the nonlinear terms $G_1 = \psi^5$ and $G_2 = v^5$. In Figure \ref{stability1_homogeneous1}, we present the results obtained for the model parameters $a= c = -2$, $b=d = 2$, $b_2=d_2 = 3$ and $a_2 = c_2 = 3$.
It is observed that the coefficient $S''(\omega)$ remains positive throughout the entire interval $[0,1]$, indicating that the corresponding traveling-wave solutions are stable for all wave velocities within this range. This extends beyond the theoretical results presented in the previous section, where stability was only established in a neighborhood of the wave velocity $\omega = 1$.

\begin{figure}[h!]
  \centering
		\includegraphics[width=9cm, height=6cm]{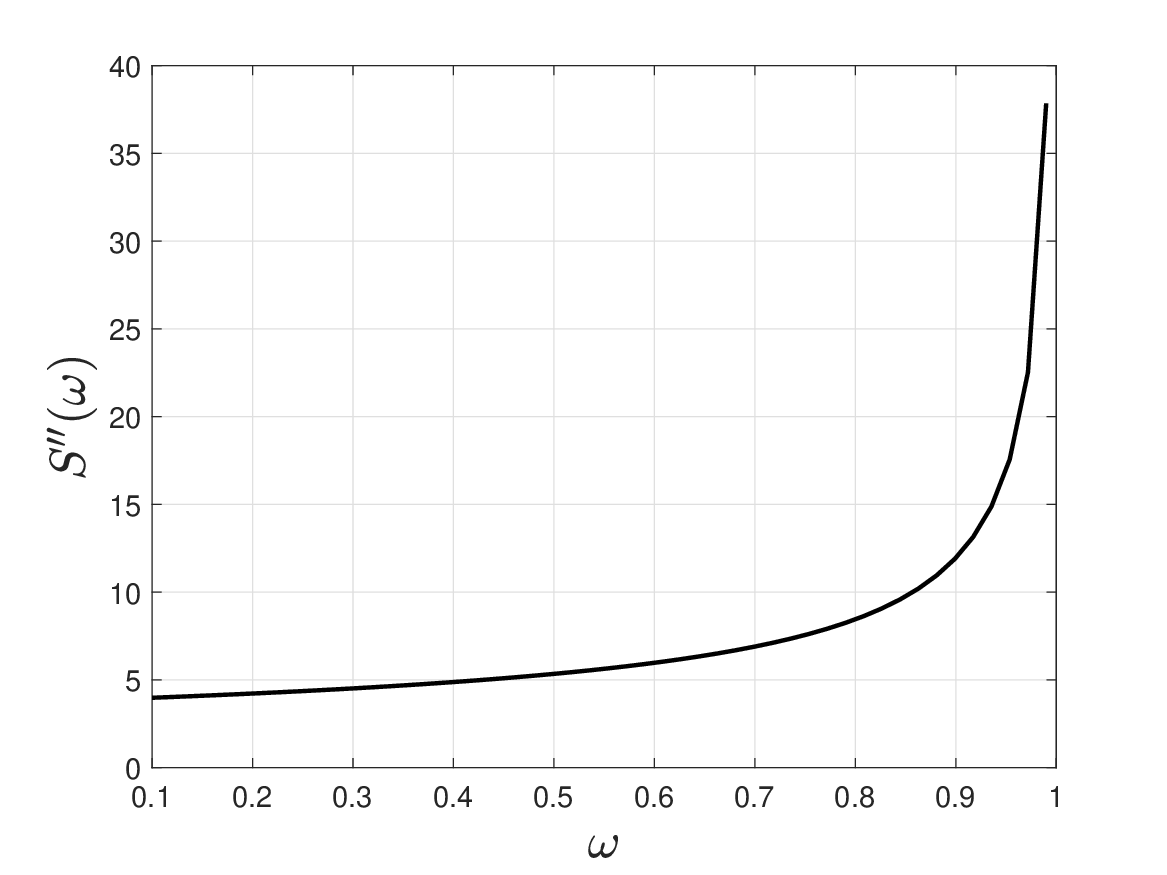}
	\caption{$S''(\omega)$ for $p=4$, $G_1 = \psi^{p+1}, G_2 = v^{p+1}$, $a= c = -2$, $b=d = 2$, $b_2=d_2 = 3.2$ and $a_2 = c_2 = 3$.}
	\label{stability1_homogeneous1}
\end{figure}

In Figure \ref{stability1_homogeneous2}, we present an additional numerical experiment with parameters $a = c = -2$, $b = d = 2$, $b_2 = d_2 = 3$, and $a_2 = c_2 = 3.2$. The nonlinear exponent is set to $p = 9$, with $G_1 = \psi^{p+1} = \psi^{10}$ and $G_2 = v^{p+1} = v^{10}$. For this value, $p > 8$, we observe that the function $S''(\omega)$ changes sign. In particular, there exists a range of wave velocities near $\omega = 1$ where $S''(\omega)$ becomes negative, which implies that the corresponding orbit $\mathcal{G}_\omega$ is unstable in this range. This experiment suggests that the stability of the orbit $\mathcal{G}$ is lost when the nonlinearity parameter $p$ exceeds the threshold value $p = 8$.

To illustrate a case with $p = 3$, $\beta=2$, which satisfies the condition $p + 2\beta < 8$, Figure \ref{stability1_homogeneous3} presents a numerical experiment for the nonlinearity
\[
F(\psi, \psi_x, v, v_x) = \psi^3 \psi_x^2 + v^3 v_x^2,
\]
with model parameters $a = c = -2$, $b = d = 2$, $b_2 = d_2 = 3$, and $a_2 = c_2 = 3.2$. In this case,
\begin{equation}
\begin{aligned}
&G_1 = -3 \psi^2 \psi_x^2 - 2 \psi_{xx} \psi^3, \\
&G_2 = -3 v^2 v_x^2 - 2 v_{xx} v^3.
\end{aligned}
\end{equation}
The numerical result shows once again that $S''(\omega)>0$ throughout the velocity range $0<\omega<\min\left\{1, \frac{-a}{b}, \frac{-c}{b},\frac{a_2}{b_2}, \frac{c_2}{b_2}\right\}$,
providing further evidence of the stability predicted by the theory.

\begin{figure}[h!]
  \centering
		\includegraphics[width=9cm, height=6cm]{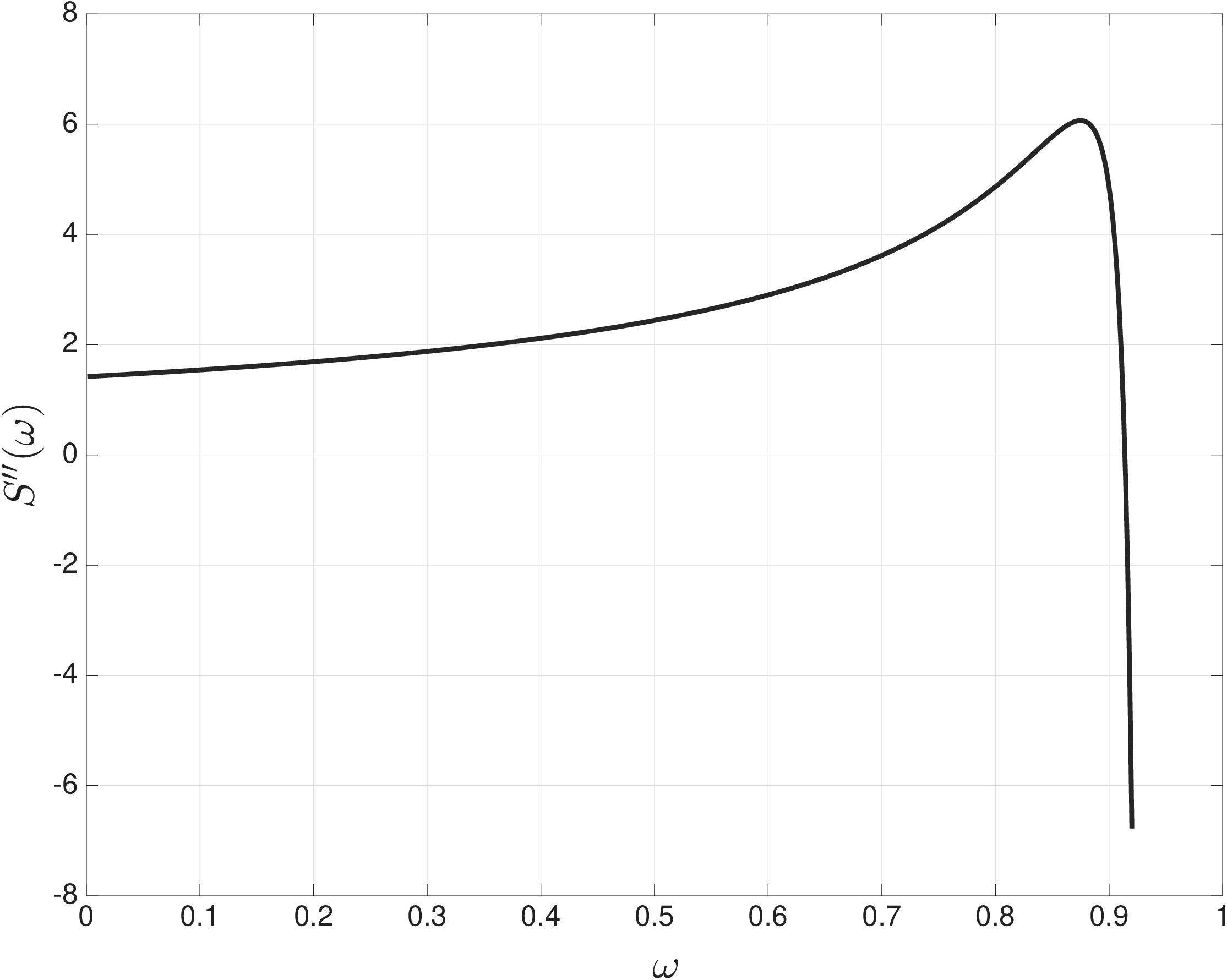}
	\caption{$S''(\omega)$ for $p=9$, $G_1 = \psi^{p+1}, G_2 = v^{p+1}$, $a= c = -2$, $b=d = 2$, $b_2=d_2 = 3.2$ and $a_2 = c_2 = 3$.}
	\label{stability1_homogeneous2}
\end{figure}

\begin{figure}[h!]
  \centering
		\includegraphics[width=9cm, height=6cm]{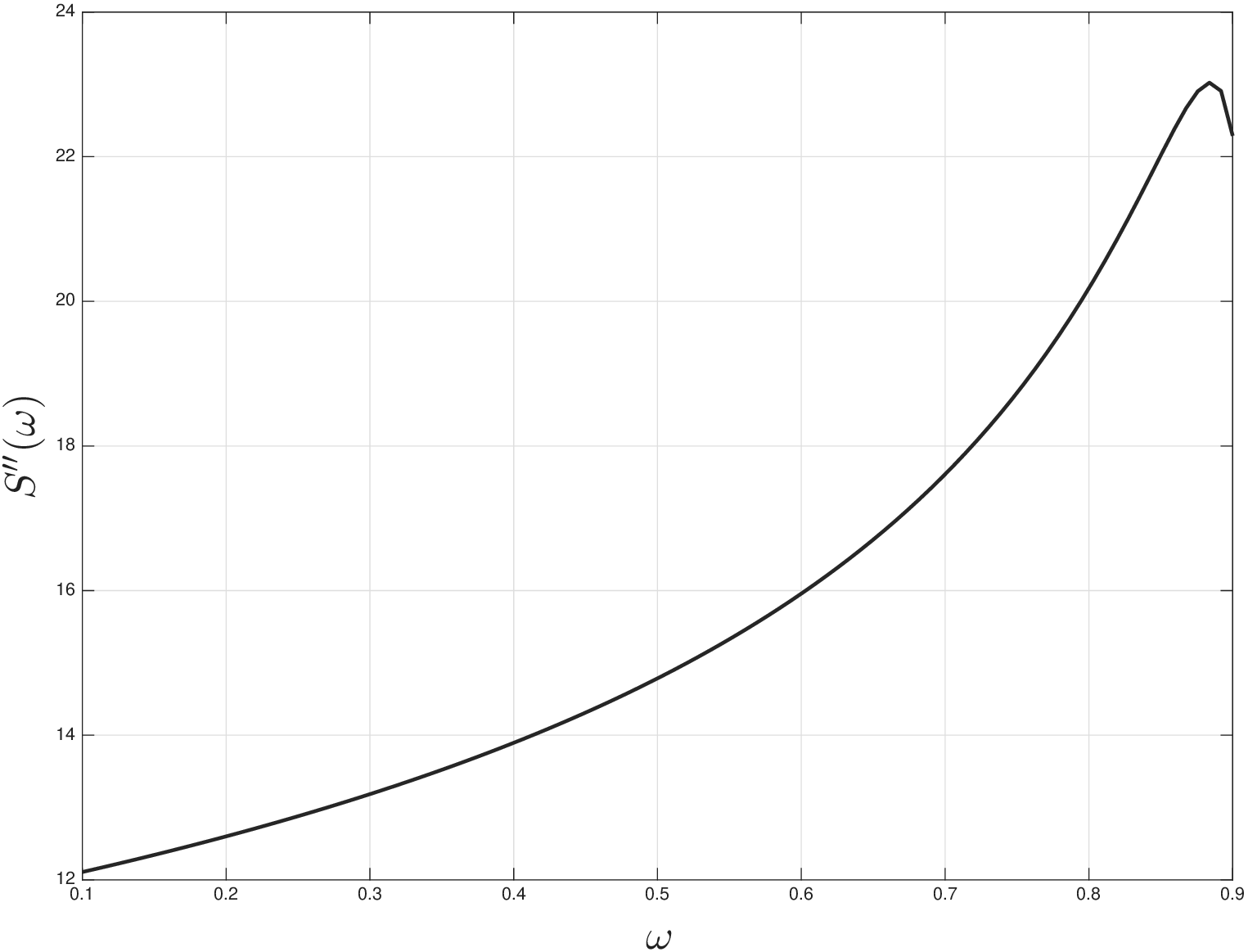}
	\caption{$S''(\omega)$ for $G_1 = -3 \psi^2 \psi_x^2 - 2 \psi_{xx} \psi^3$, $G_2 = -3 v^2 v_x^2 - 2 v_{xx} v^3$. $a= c = -2$, $b=d = 2$, $b_2=d_2 = 3.2$ and $a_2 = c_2 = 3$.}
	\label{stability1_homogeneous3}
\end{figure}

\subsection{Non-homogeneous case}

When the nonhomogeneous terms $G_1, G_2$ are not of the same order of homogeneity, the expressions in \eqref{stabilizing1}
can no longer be used.
Instead, given the values of $\hat{v}_k^s$, $\hat{\psi}_k^s$ at the iteration step $s$, we compute the stabilizing factors using, for instance, a Newton-type method. This involves solving the following nonlinear equations for $\alpha_s, \beta_s$, at each step:
\[
P_1(\alpha_s v^s) = 1, \quad \text{and}\quad P_2(\beta_s \psi^s ) = 1.
\]
Once the stabilizing parameters $\alpha_s, \beta_s$ are computed, we use the update formulas \eqref{iteration_1} and \eqref{iteration_2} to obtain the next iterates $\hat{v}_k^{s+1}$ and $\hat{\psi}_k^{s+1}$ for iteration $s+1$.

In Figure \ref{stability1_nonhomogeneous1}, we present the resulting coefficient $S''(\omega)$ corresponding to the nonlinear terms $G_1 = \psi^3$, $G_2 = v^5$, using the model parameters $a= c = -2$, $b=d=2$, $b_2 = d_2 = 3$, $a_2 = c_2 = 3$. A sign change in the function $S''(\omega)$ is observed within the interval $[0,1]$, indicating the presence of an instability in the traveling wave solutions of system \eqref{1bbl} within an interval approximately $[0, 0.55]$. In contrast, the solutions exhibit stability near the wave velocity $\omega = 1$, in agreement with the theoretical results established in the previous sections.
Further numerical evidence indicates that this region of instability becomes broader as the exponents $p$ and $q$ increase in the case $G_1 = \psi^p, G_2 = v^q$.

\begin{figure}[h!]
  \centering
		\includegraphics[width=9cm, height=6cm]{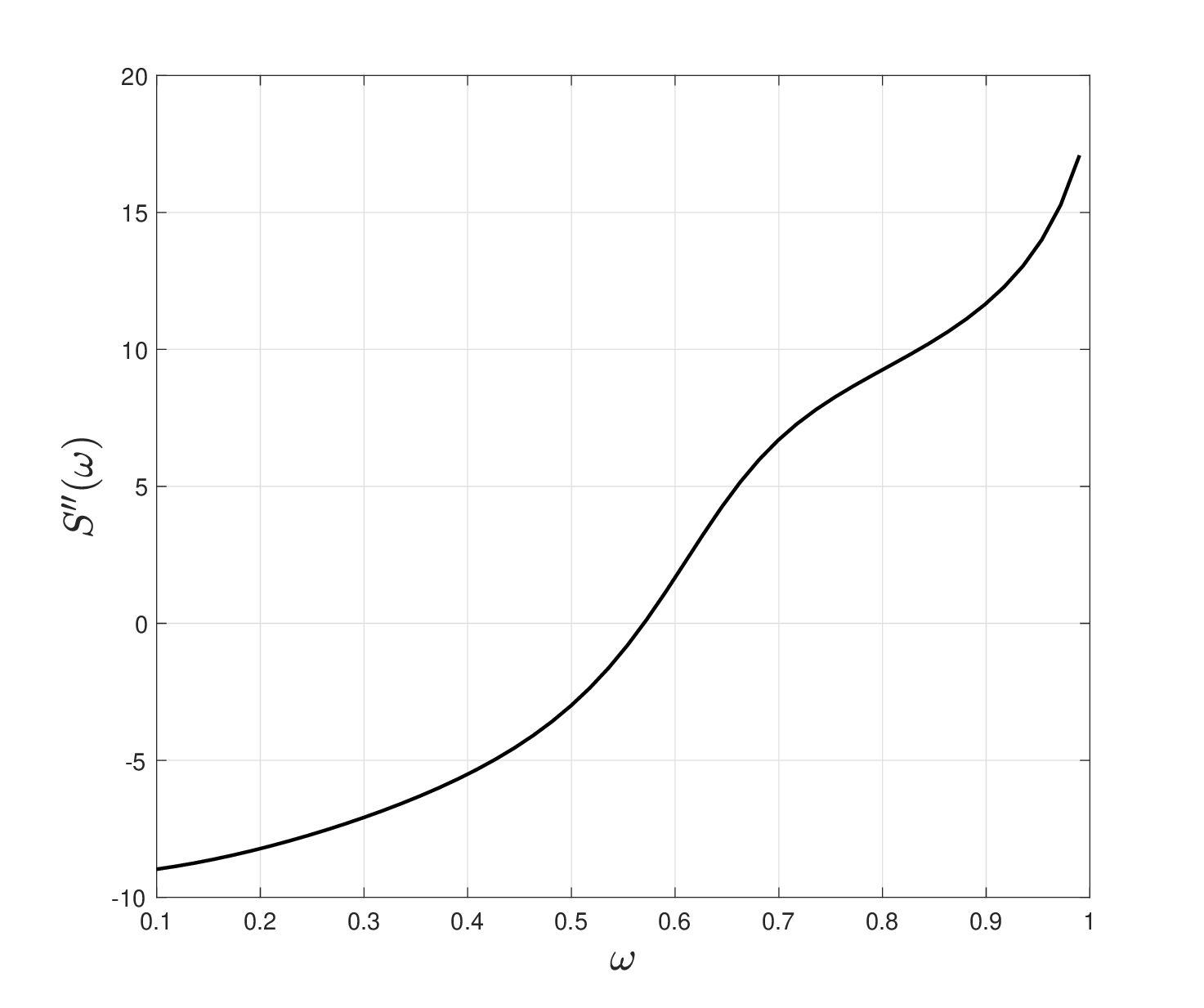}
	\caption{$S''(\omega)$ for $G_1 = \psi^3, G_2 = v^5$, $a= c = -2$, $b=d = 2$, $b_2=d_2 = 3$ and $a_2 = c_2 = 3$.}
	\label{stability1_nonhomogeneous1}
\end{figure}

Finally, in Figure \ref{stability1_nonhomogeneous2}, we consider the case where
\[
F(\psi, \psi_x, v, v_x) = \frac16 \psi^6 + (\psi_x)^2 \psi^2 + \frac16 v^6 + (v_x)^2 v^2,
\]
which leads to 
\begin{equation}
\begin{aligned}\label{nonlinear1}
&G_1 = \psi^5 - 4(\psi_x)^2 \psi - 2\psi^2 \psi_{xx},\\
&G_2 = v^5 - 4(v_x)^2 v - 2v^2 v_{xx}.
\end{aligned}
\end{equation}
We observe that the coefficient $S''(\omega)$ remains positive throughout the entire velocity range $0< \omega < 1$, indicating that the corresponding solitary wave solutions are stable throughout this interval. This numerical finding is also consistent with the theoretical results developed in the previous sections.

\begin{figure}[h!]
  \centering
		\includegraphics[width=9cm, height=6cm]{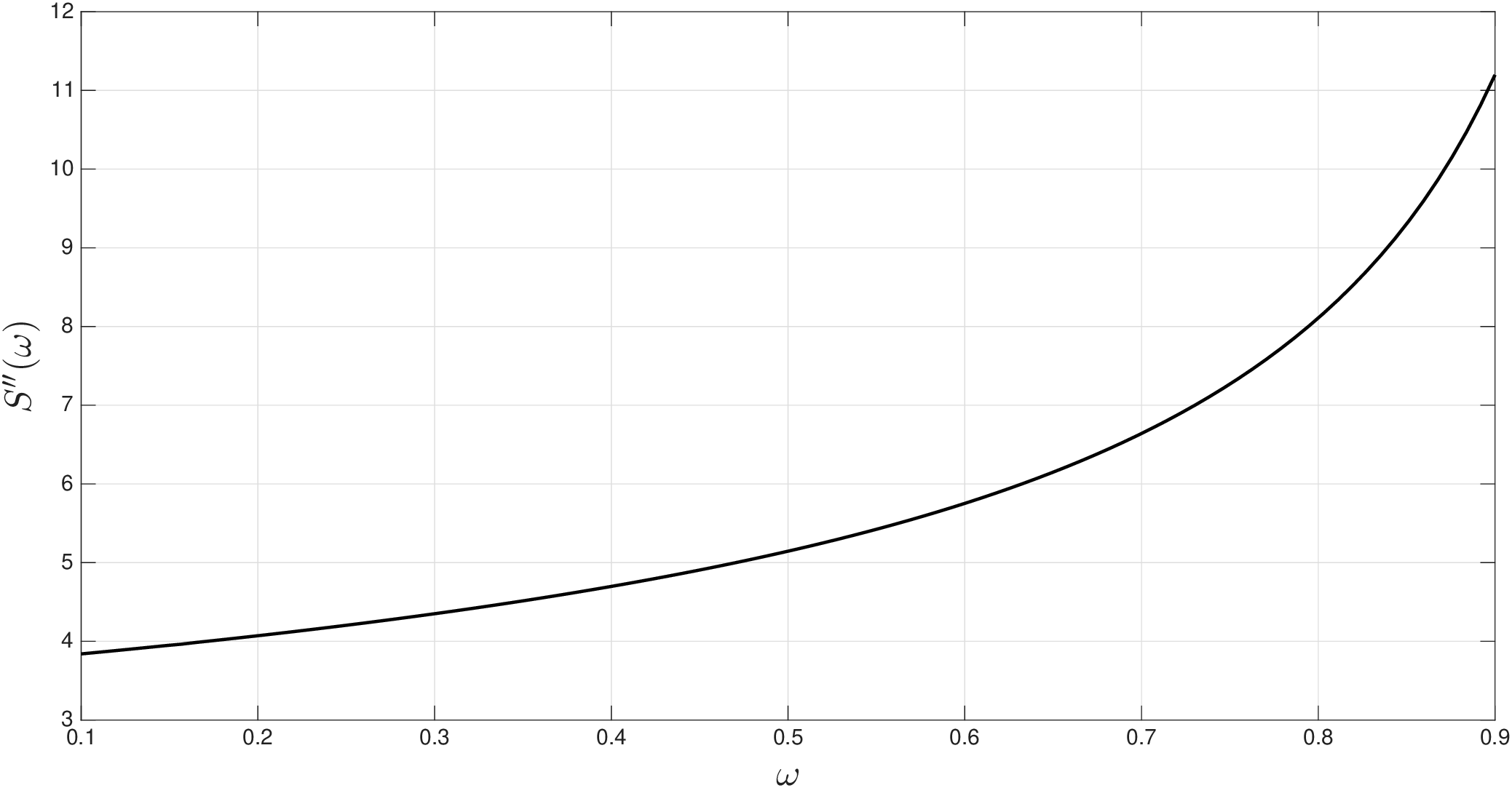}
	\caption{$S''(\omega)$ for $G_1, G_2$ given in \eqref{nonlinear1}, $a= c = -2$, $b=d = 2$, $b_2=d_2 = 2.5$ and $a_2 = c_2 = 3$.}
	\label{stability1_nonhomogeneous2}
\end{figure}

\subsection*{Acknowledgments} R. de A. Capistrano–Filho was partially supported by CAPES/CO\-FE\-CUB grant number 88887.879175/2023-00, CNPq grant numbers 301744/2025-4, 421573/2023-6, and 307808/2021-1, and PROPG (UFPE) \textit{via} PROAP resources. J. C. Mu\~noz was partially supported by Universidad del Valle (Cali, Colombia) through the research project C.I. 71409. J. R. Quintero was supported by the Mathematics Department at  Universidad del Valle  (Cali, Colombia)  under the research project C.I. 71413.

\subsection*{Data availability}  No data were used for the research described in the article.

\end{document}